%% file: paper_vc_arxiv.tex
\documentclass[10pt]{article}
\usepackage[utf8x]{inputenc}
\usepackage{graphicx}
\usepackage{float}
\usepackage{latexsym}
\usepackage{amsmath}
\usepackage{amssymb}
\usepackage{amsthm}
\usepackage{amsfonts}
\usepackage{color}

\usepackage{amscd}
\usepackage{mathrsfs}
\usepackage{bm}

\newcommand{\D}{\displaystyle}

\newcommand{\bx}{\mathbf{x}}
\newcommand{\bn}{\mathbf{n}}
\newcommand{\by}{\mathbf{y}}

\newcommand{\p}{\partial}
\newcommand{\al}{\alpha}

\newcommand{\rhkpipj}{\bar{R}_t({\bf p}_i, {\bf p}_j)}

\newcommand{\bfu}{{\bf u}}
\newcommand{\bfp}{{\bf p}}

\newcommand{\M}{{\mathcal M}}
\newcommand{\V}{{\mathcal V}}

\newcommand{\bV}{\mathbf{V}}

\newcommand{\mathd}{\mathrm{d}}

\newtheorem{theorem}{\textbf{Theorem}}[section]
\newtheorem{lemma}{\textbf{Lemma}}[section]

\newtheorem{proposition}{\textbf{Proposition}}[section]
\newtheorem{assumption}{\textbf{Assumption}}[section]

\newtheorem{corollary}{\textbf{Corollary}}[section]

\newcommand{\R}{\mathbb{R}}

\numberwithin{equation}{section}

\makeatother

\begin{document}

\title{Enforce the Dirichlet boundary condition by volume constraint in Point Integral method}

\author{
Zuoqiang Shi%
\thanks{Mathematical Sciences Center, Tsinghua University, Beijing, China,
100084. \textit{Email: zqshi@math.tsinghua.edu.cn.}%
}
}

 \maketitle

\begin{abstract}
Recently, Shi and Sun proposed Point Integral method (PIM) to discretize Laplace-Beltrami operator on point cloud \cite{LSS,SS14}. 
In PIM, Neumann boundary is nature, but Dirichlet boundary needs some special treatment. In our previous work, we use 
Robin boundary to approximate Dirichlet boundary. In this paper, we introduce another approach to deal with the Dirichlet boundary condition 
in point integral method
using the volume constraint proposed by Du et.al. \cite{Du-SIAM}.  
\end{abstract}


\section{Introduction}
\label{sec:intro}

Partial differential equations on manifold appear in
a wide range of applications such as material science \cite{CFP97,EE08}, fluid flow \cite{GT09,JL04}, 
biology and biophysics \cite{BEM11,ES10,NMWI11,WD08} and machine learning and data analysis \cite{belkin2003led, Coifman05geometricdiffusions}. 
Due to the complicate geometrical structure of the manifold, it is very chanlleging to 
solve PDEs on manifold.
In recent years, it attracts more and more attentions to develop efficient numerical method to solve 
PDEs on manifold. In case of that the manifold is a 2D surface embedding in $\mathbb{R}^3$, 
many methods were proposed include level set methods \cite{AS03,XZ03},
surface finite elements \cite{DE-Acta}, finite volume methods \cite{LNR11}, diffuse interface methods \cite{ESSW11} and local mesh methods \cite{Lai13}.

In this paper, we focus on following Poisson equation with Dirichlet boundary condition
\begin{align} \label{eq:dirichlet}
\left \{
\begin{array}{rlll}
    -\Delta_{\mathcal{M}} u(\bx)&=& f(\bx), \quad & \bx \in \mathcal{M} \\
    u(\bx) & = & 0,\quad & \bx \in \p\mathcal{M} \\
\end{array}
\right.
\end{align}
where $\M$ is a smooth manifold isometrically embedded in $\R^d$ with the standard Euclidean metric 
and $\p\M$ is the boundary. $\Delta_\mathcal{M}$ is the Laplace-Beltrami operator on manifold $\mathcal{M}$.  
Let $g$ be the Riemannian metric tensor of $\mathcal{M}$. 
Given a local coordinate system $(x^1, x^2, \cdots, x^k)$,
the metric tensor $g$ can be represented by a matrix $[g_{ij}]_{k\times k}$,
\begin{eqnarray}
  g_{ij}=<\frac{\p}{\p x^i},\frac{\p}{\p x^j}>,\quad i,j=1,\cdots,k.\nonumber
\end{eqnarray}
Let $[g^{ij}]_{k\times k}$ is the inverse matrix of $[g_{ij}]_{k\times k}$, then it is well known that
the Laplace-Beltrami operator is
\begin{equation}
\Delta_{\mathcal{M}}  = \frac{1}{\sqrt{\det g}}\frac{\p}{\p x^i}(g^{ij}\sqrt{\det g} \frac{\p}{\p x^j}).\nonumber
\end{equation}
In this paper, the metric tensor $g$ is assumed to be inherited from the ambient space $\R^d$, that is, 
$\M$ isometrically embedded in $\R^d$ with the standard Euclidean metric.
If $ \mathcal{M}$ is an open set in $\R^d$, 
then $\Delta_\mathcal{M}$ becomes standard Laplace operator, i.e., $\Delta_{ \mathcal{M}} = \sum_{i=1}^d \frac{\p^2 }{\p {x^i}^2}$.

In our previous papers, \cite{LSS,SS14},  Point Integral method was developed to solve Poisson equation in point   
cloud. The main observation of the Point Integral method is that the solution of the Poisson equation can be approximated
by an integral equation,
\begin{equation}
  \label{eq:integral-pim}
  \frac{1}{t}\int_{\M}R_t(\bx,\by)(u(\bx)-u(\by))\mathd\by-2\int_{\p\M}\bar{R}_{t}(\bx,\by)\frac{\p u}{\p \bn}(\by)\mathd\mu_\by=
\int_{\M}\bar{R}_t(\bx,\by)f(\by)\mathd\by
\end{equation}
where $\bn$ is the out normal of $\M$ at $\p\M$. The kernel functions
\begin{eqnarray}
  \label{eq:kernel}
 R_t(\bx,\by)=\frac{1}{(4\pi t)^{k/2}}R\left(\frac{\|\bx-\by\|^2}{4t}\right),\quad  \bar{R}_t(\bx,\by)=\frac{1}{(4\pi t)^{k/2}}\bar{R}\left(\frac{\|\bx-\by\|^2}{4t}\right)
\end{eqnarray}
and $\bar{R}(r)=\int_{r}^{+\infty}R(s)\mathd s$. $t$ is a parameter, which is determined by the desensity of the point cloud in the real computations. 

The kernel function $R(r): \R^+ \rightarrow \R^+ $ is assumed to be $C^2$ smooth and satisfies some mild conditions (see Section \ref{sec:PIM}).

The integral approximation \eqref{eq:integral-pim} is natural to solve the Poisson equation with Neumann boundary condition. To enforce the Dirichlet boundary 
condition, in our previous work \cite{LSS,SS14}, we used Robin boundary condition to approximate the Dirichlet boundary condition. More specifically, we solve 
following problem instead of \eqref{eq:dirichlet} with $0<\beta\ll 1$,
\begin{align} \label{eq:robin}
\left \{
\begin{array}{rlll}
    -\Delta_{\mathcal{M}} u(\bx)&=& f(\bx), \quad & \bx \in \mathcal{M}, \\
    u(\bx)+\beta\frac{\p u}{\p \bn} & = & 0,\quad & \bx \in \p\mathcal{M}. \\
\end{array}
\right.
\end{align}
Using \eqref{eq:integral-pim}, we have an integral equation to approximate the above Robin problem,
\begin{equation}
  \label{eq:integral-pim-robin}
  \frac{1}{t}\int_{\M}R_t(\bx,\by)(u(\bx)-u(\by))\mathd\by+\frac{2}{\beta}\int_{\p\M}\bar{R}_{t}(\bx,\by)u(\by)\mathd\mu_\by=
\int_{\M}\bar{R}_t(\bx,\by)f(\by)\mathd\by.
\end{equation}
We can prove that this approach converge to the original Dirichlet problem \cite{SS14}. In the real computations, small $\beta$ may give some trouble. 
The overcome this problem, we also introduced an itegrative method to enforce the Dirichlet boundary condition based on the Augmented Lagrangian Multiplier 
(ALM) method. However, we can not prove the convergence of this iterative method, although it always converges in the numerical tests.

Recently, Du et.al. \cite{Du-SIAM} proposed volume constraint to deal with the boundary condition in the nonlocal diffusion problem. They found that in the 
nonlocal diffusion problem, 
since the operator is nonlocal, only enforce the boundary condition on the boundary is not enough, we have to extend the boundary condition to 
a small region close to the boundary. 
Borrowing this idea, in nonlocal diffusion problem  to handle the Dirichlet boundary. This idea gives us following integral equation with volume constraint:
\begin{align} \label{eq:integral}
\left \{
\begin{array}{rlll}
 \D   \frac{1}{t}\int_{\M}R_t(\bx,\by)(u(\bx)-u(\by))\mathd\by&=&\D\int_{\M}\bar{R}_t(\bx,\by)f(\by)\mathd\by, 
\quad & \bx \in \mathcal{M}'_t \\
    u(\bx) & = & 0,\quad & \bx \in \mathcal{V}_t \\
\end{array}
\right.
\end{align}
Here, $\M'_t$ and $\V_t$ are subsets of $\M$ which are defined as
\begin{eqnarray}
  \label{eq:domain}
  \M'_t=\left\{\bx\in \M: B\left(\bx,2\sqrt{t}\right)\cap\p\M=\emptyset\right\},\quad \mathcal{V}_t=\M\backslash\M'_t.
\end{eqnarray}
The thickness of $\mathcal{V}_t$ is $2\sqrt{t}$ which implies that $|\mathcal{V}_t|=O(\sqrt{t})$. The relation of $\M$, $\p\M$, $\M'_t$ and $\V_t$ are sketched in 
Fig. \ref{fig:domain}.
\begin{figure}
  \centering
  \includegraphics[width=0.6\textwidth]{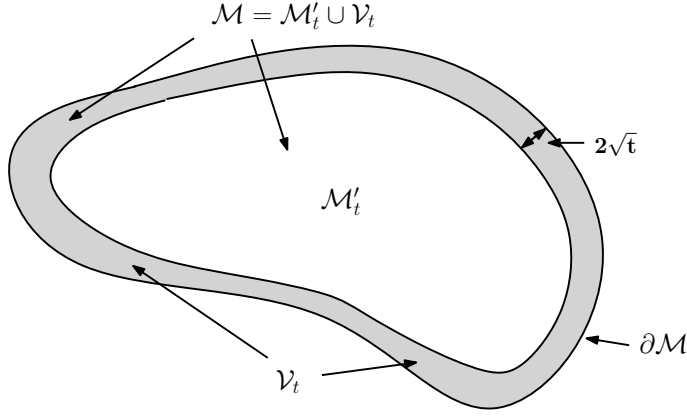}
  \caption{Computational domain for volume constraint}
  \label{fig:domain}
\end{figure}

The main advantage of the integral equation \eqref{eq:integral} is that there is not any differential operator in the integral equation. Then it is easy 
to discretized on point cloud. 
Assume we are given
a set of sample points $P=\{\bfp_i:\bfp_i\in \M,\;i=1,\cdots,n\}$ sampling the submanifold $\M$ and
one vector $\bV = (V_1, \cdots, V_n)^t$ where $V_i$ is the volume weight of $\bfp_i$ in $\M$. In addition, we assume that the 
point set $P$ is a good sample of manifold $\M$ in the sense that the integral on $\M$ can be well approximated by the summation over $P$, see
Section \ref{sec:PIM}. 

Then, \eqref{eq:integral} can be easily discretized to get following linear system
\begin{eqnarray}
  \label{eq:dis}
\left\{\begin{array}{rclc}
  \D\frac{1}{t}\sum_{\bfp_j\in\M}R_t(\bfp_i,\bfp_j)(u_i-u_j)V_j
&=&\D\sum_{\bfp_j\in \M}\bar{R}_t(\bfp_i,\bfp_j)f(\bfp_j)V_j,&  
\bfp_i \in \mathcal{M}'_t,\\ 
u_i&=&0,&\bfp_i\in \V_t.
\end{array}
\right.
\end{eqnarray}
This is the discretization of the Poisson equation \eqref{eq:dirichlet} given by Point Integral Method with volume constraint on point cloud. 

Similarly, the eigenvalue problem
\begin{align} \label{eq:eigen}
\left \{
\begin{array}{rlll}
    -\Delta_{\mathcal{M}} u(\bx)&=& \lambda u, \quad & \bx \in \mathcal{M} \\
    u(\bx) & = & 0,\quad & \bx \in \p\mathcal{M} \\
\end{array}
\right.
\end{align}
can be approximated by an integral eigenvalue problem
\begin{align} \label{eq:integral-eigen}
\left \{
\begin{array}{rlll}
 \D   \frac{1}{t}\int_{\M}R_t(\bx,\by)(u(\bx)-u(\by))\mathd\by&=&\D\lambda \int_{\M}\bar{R}_t(\bx,\by)u(\by)\mathd\by, 
\quad & \bx \in \mathcal{M}'_t \\
    u(\bx) & = & 0,\quad & \bx \in \mathcal{V}_t \\
\end{array}
\right.
\end{align}
And corresponding discretization is given as following
\begin{eqnarray}
  \label{eq:dis-eigen}
\quad\quad\quad\left\{\begin{array}{rclc}
  \D\frac{1}{t}\sum_{\bfp_j\in\M}R_t(\bfp_i,\bfp_j)(u_i-u_j)V_j
&=&\D\lambda \sum_{\bfp_j\in \M}\bar{R}_t(\bfp_i,\bfp_j)u_jV_j,&  
\bfp_i \in \mathcal{M}'_t,\\ 
u_i&=&0,&\bfp_i\in \V_t.
\end{array}
\right.
\end{eqnarray}

\subsection{Assumptions and main results}
\label{sec:PIM}
One of the main contribution of this paper is that, under some assumptions, 
we prove that the solution of the discrete system \eqref{eq:dis} converges to the solution of the Poisson equation 
\eqref{eq:dirichlet} and the spectra of the eigen problem \eqref{eq:dis-eigen} converge to the spectra of the Laplace-Beltrami operator with 
Dirichlet boundary \eqref{eq:eigen}. 

The assumptions we used are listed as following.
\begin{assumption} 
\begin{itemize}
\item Assumptions on the manifold: $\M, \p\M$ are both compact and $C^\infty$ smooth.
\item Assumptions on the sample points $(P,\mathbf{V})$: $(P,\mathbf{V})$ is $h$-integrable approximation of $\M$, i.e.
  \begin{itemize}
  \item[]  For any function $f\in C^1(\M)$, there is a constant $C$ 
independent of $h$ and $f$ so that
\begin{equation*}
\left|\int_\M f(\by) \mathd\by - \sum_{\bfp_i\in\M} f(\bfp_i)V_i\right| < Ch|\text{supp}(f)|\|f\|_{C^1(\M)}.
\end{equation*}
  \end{itemize}
\item Assumptions on the kernel function $R(r)$:
\begin{itemize}
\item[(a)] $R\in C^2(\mathbb{R}^+)$;
\item[(b)]
$R(r)\ge 0$ and $R(r) = 0$ for $\forall r >1$; 
\item[(c)]
 $\exists \delta_0>0$ so that $R(r)\ge\delta_0$ for $0\le r\le\frac{1}{2}$.
\end{itemize}
\end{itemize}
\end{assumption}
These assumptions are default in this paper and they are omitted in the statement of the theoretical results. And in the analysis, we always assume that 
$t$ and $h/\sqrt{t}$ are small enough. Here, "small enough" means that they are less than a generic constant which only depends on $\M$.

Under above assumptions, we have two theorems regarding the convergence of the Poisson equation and corresponding eigenvaule problem.
\begin{theorem}
  \label{thm:converge}
Let $u(\bx)$ be solution of \eqref{eq:dirichlet} and $\bfu=[u_1,\cdots,u_n]^t$ be solution of \eqref{eq:dis} and $f\in C^1(\M)$ in both problems. 
There exists $C>0$ only depends on $\M$ and $\p\M$, such that
\begin{eqnarray*}
  \|u-u_{t,h}\|_{H^1(\M'_t)}\le C\left(t^{1/4}+\frac{h}{t^{3/2}}\right)\|f\|_{C^1(\M)}
\end{eqnarray*}
where 
\begin{equation}
  \label{eq:interpolation-intro}
  u_{t,h}(\bx)=\left\{\begin{array}{cc}\D
\frac{1}{w_{t,h}(\bx)}\left(\sum_{\bfp_j\in\M}R_t(\bx,\bfp_j)u_jV_j
 +t\sum_{\bfp_j\in\M}\bar{R}_t(\bx,\bfp_j)f(\bfp_j)V_j\right),&\bx\in \M'_t,\\
0,&\bx\in \V_t.
\end{array}\right.
\end{equation}
and $w_{t,h}(\bx)=\sum_{\bfp_j\in \M}R_t(\bx,\bfp_j)V_j$.
\end{theorem}

\begin{theorem}
Let $\lambda_i$ be the $i$th largest eigenvalue of eigenvalue problem \eqref{eq:eigen}. And
let $\lambda_{i}^{t,h}$ be the $i$th largest eigenvalue of
discrete eigenvalue problem \eqref{eq:dis-eigen}, then there exists a constant $C$ such that
$$|\lambda_i^{t,h} - \lambda_i|\le C \lambda_i^2 \left(t^{1/4}+\frac{h}{t^{d/4+3}}\right) , $$
and there exist another constant $C$ such that, for any $\phi\in E(\lambda_i,T)X$ and $X=H^1(\M'_t)$,
$$\|\phi-E(\sigma_{i}^{t,h},T_{t,h})\phi\|_{H^1(\M'_t)} \le C \left(t^{1/4}+\frac{h}{t^{d/4+2}}\right) .$$
where $\sigma_i^{t,h}=\{\lambda_j^{t,h}\in \sigma(T_{t,h}): j\in I_i\}$ and $I_i=\{j\in \mathbb{N}: \lambda_j=\lambda_i\}$,
$E(\lambda,T)$ is the Riesz spectral projection associated with
$\lambda$.
\label{thm:eigen_converge}
\end{theorem}

\section{Stability analysis}
\label{sec:stability}
To prove the convergence, we need some stability results which are listed in this section. 
The first lemma is about the coercivity of the integral operator and the proof can be found in \cite{SS14}. 
\begin{lemma}
\label{lem:coercivity}
 For any function $u\in L^2(\mathcal{M})$, 
there exists a constant $C>0$ only depends on $\M$, such that
  \begin{eqnarray}
    \int_{\mathcal{M}}\int_{\mathcal{M}} R_t(\bx,\by)(u(\bx)-u(\by))^2\mathd\bx\mathd\by \ge C\int_\mathcal{M} |\nabla v|^2\mathd \bx,\nonumber
  \end{eqnarray}
where 
\begin{eqnarray}
v(\bx)=\frac{1}{w_t(\bx)}\int_{\mathcal{M}}R_t(\bx,\by)u(\by)\mathd \by, \nonumber
\end{eqnarray}
and $\D w_t(\bx) = \int_{\mathcal{M}}R_t\left(\bx,\by\right)\mathd \by$.
\end{lemma}
Next corollary directly follows from Lemma \ref{lem:coercivity}.
\begin{corollary}
\label{cor:coercivity-inner}
For any function $u\in L_2(\M_t')$, 
there exists a constant $C>0$ only depneds on $\M$, such that
 \begin{eqnarray}
    \frac{1}{t}\int_{\mathcal{M}'_t}\int_{\mathcal{M}'_t} R_t(\bx,\by)(u(\bx)-u(\by))^2\mathd\bx\mathd\by
+\frac{1}{t}\int_{\M'_t}u^2(\bx)\left(\int_{\V_t} R_t(\bx,\by)\mathd \by\right)\mathd\bx\ge C\int_{\mathcal{M}'_t} |\nabla v|^2\mathd \bx,\nonumber
  \end{eqnarray}
where 
\begin{eqnarray}
v(\bx)=\frac{1}{w_t(\bx)}\int_{\mathcal{M}'_t}R_t(\bx,\by)u(\by)\mathd \by, \nonumber
\end{eqnarray}
and $\D w_t(\bx) = \int_{\mathcal{M}}R_t\left(\bx,\by\right)\mathd \by$.
\end{corollary}
\begin{proof}
  Let
  \begin{eqnarray*}
    \tilde{u}(\bx)=\left\{\begin{array}{cc}
u(\bx),&\bx\in \M'_t,\\
0,& \bx\in \V_t.
\end{array}\right.
  \end{eqnarray*}
Using Lemma \ref{lem:coercivity},
\begin{eqnarray}
 && \int_{\mathcal{M}'_t} |\nabla v|^2\mathd \bx\le \int_{\mathcal{M}} |\nabla v|^2\mathd \bx\nonumber\\
&\le& \frac{C}{t}\int_{\mathcal{M}}\int_{\mathcal{M}} R_t(\bx,\by)(\tilde{u}(\bx)-\tilde{u}(\by))^2\mathd\bx\mathd\by\nonumber\\
&=& \frac{C}{t}\int_{\mathcal{M}'_t}\int_{\mathcal{M}'_t} R_t(\bx,\by)(u(\bx)-u(\by))^2\mathd\bx\mathd\by
+\frac{C}{t}\int_{\M'_t}u^2(\bx)\left(\int_{\V_t} R_t(\bx,\by)\mathd \by\right)\mathd\bx.\nonumber
\end{eqnarray}
\end{proof}

Using Lemma \ref{lem:coercivity}, we can also get following lemma regarding the stability in $L^2(\M)$.
\begin{lemma}
For any function $u\in L_2(\M)$ with $ u(\bx) = 0$ in $\V_t$, 
there exists a constant $C>0$ independent on $t$
\begin{eqnarray}
\frac{1}{t}\int_{\M}\int_{\M}R_t(\bx,\by)(u(\bx)-u(\by))^2\mathd\bx\mathd\by
\geq C\|u\|_{L_2(\M)}^2,\nonumber
\end{eqnarray}
as long as $t$ small enough.
\label{lem:elliptic_L_t}
\end{lemma} 

\begin{proof}
Let
  \begin{eqnarray}
    v(\bx)=\frac{1}{w_t(\bx)}\int_\M R_t(\bx,\by)u(\by)\mathd \by.\nonumber
  \end{eqnarray}
Since $u(\bx)=0,\; \bx\in \V_t$, we have
\begin{eqnarray}
  \label{eq:v-boundary}
  v(\bx)=0,\quad \bx\in \p\M.\nonumber
\end{eqnarray}
By Lemma~\ref{lem:coercivity} and the Poincare inequality, there exists a constant $C>0$, such that 
  \begin{eqnarray}
    \int_{\mathcal{M}} |v(\bx)|^2\mathd \bx \le  \int_{\mathcal{M}} |\nabla v(\bx)|^2\mathd \bx&\le& \frac{C}{t}\int_{\mathcal{M}}\int_{\mathcal{M}} R_t\left(\bx,\by\right)  (u(\bx)-u(\by))^2\mathd \mu_\bx \mathd \mu_\by \nonumber
  \end{eqnarray}

Let $\delta=\frac{w_{\min}}{2w_{\max}+w_{\min}}$. If $u$ is smooth and close to its smoothed version $v$, in particular, 
\begin{eqnarray}
  \int_\mathcal{M} v^2(\bx)\mathd \mu_\bx\ge \delta^2 \int_\mathcal{M} u^2(\bx)\mathd \mu_\bx,
\label{eqn:smooth_u}
\end{eqnarray}
then the proof is completed. 

Now consider the case where $\eqref{eqn:smooth_u}$ does not hold.  Note that we now have
\begin{eqnarray}
  \|u-v\|_{L^2(\M)} &\ge& \|u\|_{L_2(\M)}-\|v\|_{L_2(\M)}> (1-\delta)\|u\|_{L_2(\M)}\nonumber\\
  &>&\frac{1-\delta}{\delta}\|v\|_{L_2(\M)}=\frac{2w_{\max}}{w_{\min}}\|v\|_{L_2(\M)}.\nonumber 
\end{eqnarray}
Then we have  
\begin{eqnarray}
  &&\frac{C_t}{t}\int_\mathcal{M}\int_{\mathcal{M}} R\left(\frac{|\bx-\by|^2}{4t}\right) (u(\bx)-u(\by))^2\mathd \mu_\bx \mathd \mu_\by\nonumber\\
&=&\frac{2C_t }{t}\int_\mathcal{M} u(\bx)\int_{\mathcal{M}} R\left(\frac{|\bx-\by|^2}{4t}\right) (u(\bx)-u(\by))\mathd \mu_\by \mathd \mu_\bx\nonumber\\
&=&\frac{2 }{t}\left(\int_\mathcal{M} u^2(\bx)w_t(\bx)\mathd \mu_\bx-\int_{\mathcal{M}}u(\bx)v(\bx)w_t(\bx)\mathd \mu_\bx\right)\nonumber\\
&=&\frac{2 }{t}\left(\int_\mathcal{M} (u(\bx)-v(\bx))^2w_t(\bx)\mathd \mu_\bx+
\int_{\mathcal{M}}(u(\bx)-v(\bx))v(\bx)w_t(\bx)\mathd \mu_\bx\right)\nonumber\\
&\ge & \frac{2 }{t}\int_\mathcal{M} (u(\bx)-v(\bx))^2w_t(\bx)\mathd \mu_\bx-\frac{2 }{t}
\left(\int_{\mathcal{M}}v^2(\bx)w_t(\bx)\mathd \mu_\bx\right)^{1/2}
\left(\int_{\mathcal{M}}(u(\bx)-v(\bx))^2w_t(\bx)\mathd \mu_\bx\right)^{1/2}\nonumber\\
&\ge& \frac{2 w_{\min}}{t}\int_\mathcal{M} (u(\bx)-v(\bx))^2\mathd \mu_\bx-\frac{2 w_{\max}}{t}
\left(\int_{\mathcal{M}}v^2(\bx)\mathd \mu_\bx\right)^{1/2}
\left(\int_{\mathcal{M}}(u(\bx)-v(\bx))^2\mathd \mu_\bx\right)^{1/2}\nonumber\\
&\ge&\frac{ w_{\min}}{t}\int_\mathcal{M} (u(\bx)-v(\bx))^2\mathd \mu_\bx\ge \frac{ w_{\min}}{t}(1-\delta)^2\int_\mathcal{M} u^2(\bx)\mathd \mu_\bx.\nonumber
\end{eqnarray} 
This completes the proof for the theorem.  
\end{proof}

\begin{corollary}
  \label{cor:l2-inner}
 For any function $u\in L_2(\M_t')$, 
there exists a constant $C>0$ independent on $t$, such that
\begin{eqnarray}
\frac{1}{t}\int_{\M'_t}\int_{\M'_t}R_t(\bx,\by)(u(\bx)-u(\by))^2\mathd\bx\mathd\by
+\int_{\M'_t}u^2(\bx)\left(\int_{\V_t}R_t(\bx,\by)\mathd \by\right)\mathd\bx 
\geq C\|u\|_{L_2(\M'_t)}^2,\nonumber
\end{eqnarray}
as long as $t$ small enough.
\end{corollary}
\begin{proof}
  Consider 
\begin{eqnarray*}
    \tilde{u}(\bx)=\left\{\begin{array}{cc}
u(\bx),&\bx\in \M'_t,\\
0,& \bx\in \V_t.
\end{array}\right.
  \end{eqnarray*}
and apply Lemma \ref{lem:elliptic_L_t}.
\end{proof}
Now, we can prove one important theorem.
\begin{theorem}
  \label{thm:stability}
Let $u(\bx)\in L^2(\M)$ be solution of following integral equation
\begin{eqnarray} 
\left \{
\begin{array}{rlll}
 \D   \frac{1}{t}\int_{\M}R_t(\bx,\by)(u(\bx)-u(\by))\mathd\by&=&\D r(\bx), 
\quad & \bx \in \mathcal{M}'_t \\
    u(\bx) & = & 0,\quad & \bx \in \mathcal{V}_t \\
\end{array}
\right.
\label{eq:integral-equation}
\end{eqnarray}
There exists $C>0$ only depends on $\M$ and $\p\M$, 
such that
\begin{eqnarray*}
  \|u\|_{H^1(\M'_t)}\le C\|r\|_{L^2(\M'_t)}+Ct\|\nabla r\|_{L^2(\M'_t)}
\end{eqnarray*}
\end{theorem}
\begin{proof} 
First of all, we have
\begin{eqnarray}
&&  \frac{1}{t}\int_{\M'_t}u(\bx)\int_{\M}R_t(\bx,\by)(u(\bx)-u(\by))\mathd\by\mathd\bx\nonumber\\
&=& \frac{1}{t}\int_{\M'_t}u(\bx)\int_{\M'_t}R_t(\bx,\by)(u(\bx)-u(\by))\mathd\by\mathd\bx+
 \frac{1}{t}\int_{\M'_t}u(\bx)\int_{\V_t}R_t(\bx,\by)(u(\bx)-u(\by))\mathd\by\mathd\bx\nonumber\\
&=& \frac{1}{2t}\int_{\M'_t}\int_{\M'_t}R_t(\bx,\by)(u(\bx)-u(\by))^2\mathd\bx\mathd\by+\frac{1}{t}\int_{\M'_t}u^2(\bx)
\int_{\V_t}R_t(\bx,\by)\mathd\by\mathd\bx.
\nonumber
\end{eqnarray}

 Now we can get $L^2$ estimate of $u$.  
Using Corollary \ref{cor:l2-inner}, we have 
\begin{eqnarray}
  \|u\|_{2,\M'_t}^2&\le& \frac{C}{t}\int_{\M'_t}\int_{\M'_t}R_t(\bx,\by)(u(\bx)-u(\by))^2\mathd\bx\mathd\by
+\frac{C}{t}\int_{\M'_t}|u(\bx)|^2\left(\int_{\V_t}R_t(\bx,\by)\mathd\by\right)\mathd\bx\nonumber\\
&\le&  \left| \frac{C}{t}\int_{\M'_t}u(\bx)\int_{\M}R_t(\bx,\by)(u(\bx)-u(\by))\mathd\by\mathd\bx\right|\nonumber\\
&\le&C\|u\|_{2,\M'_t}\|r\|_{2,\M'_t}\nonumber
\end{eqnarray}
This gives that
\begin{eqnarray}
\label{eq:error-l2-stable}
  \|u\|_{L^2(\M'_t)}\le C\|r\|_{L^2(\M'_t)}.
\end{eqnarray}

Next, we turn to estimate the $L^2$ norm of $\nabla e_t$ in $\M'_t$. 
Using the integral equation \eqref{eq:integral-equation}, $u$ has following expression
\begin{eqnarray}
  u_t(\bx)&=&\frac{1}{w_t(\bx)}\int_{\M'_t}R_t(\bx,\by)u_t(\by)\mathd\by+\frac{t}{w_t(\bx)}r(\bx),\quad \bx \in \V_t.
\end{eqnarray}
Then $\|\nabla u_t\|_{2,\M'_t}^2$ can be bounded as following
\begin{equation}
  \label{eq:error-grad-stable}
  \|\nabla u_t\|_{2,\M'_t}^2\le C\left\|\nabla\left(\frac{1}{w_t(\bx)}\int_{\M'_t}R_t(\bx,\by)u_t(\by)\mathd\by\right)\right\|_{2,\M'_t}^2+Ct^2\left\|\nabla\left(\frac{r(\bx)}{w_t(\bx)}\right)\right\|_{2,\M'_t}^2
\end{equation}
Corollary \ref{cor:coercivity-inner} gives a bound the first term of \eqref{eq:error-grad-stable}.
\begin{align}
  \label{eq:error-dl2-1-stable}
&  \left\|\nabla\left(\frac{1}{w_t(\bx)}\int_{\M'_t}R_t(\bx,\by)u_t(\by)\mathd\by\right)\right\|_{2,\M'_t}^2\\
\le& 
\frac{C}{t}\int_{\M'_t}\int_{\M'_t}R_t(\bx,\by)(u_t(\bx)-u_t(\by))^2\mathd\by\mathd\bx
+\frac{C}{t}\int_{\M'_t}|u_t(\bx)|^2\left(\int_{\V_t}R_t(\bx,\by)\mathd\by\right)\mathd\bx.\nonumber
\end{align}

The second terms of \eqref{eq:error-grad-stable} can be bounded by direct calculation.
\begin{eqnarray}
  \label{eq:error-dl2-3-stable}
  \left\|\nabla\left(\frac{r(\bx)}{w_t(\bx)}\right)\right\|_{2,\M'_t}^2&\le& C\left\|\frac{\nabla r(\bx)}{w_t(\bx)}\right\|_{2,\M'_t}^2
+C\left\|\frac{ r(\bx)\nabla w_t(\bx)}{(w_t(\bx))^2}\right\|_{2,\M'_t}^2\\
&\le& C\left\|\nabla r(\bx)\right\|_{2,\M'_t}^2
+\frac{C}{t}\left\|r(\bx)\right\|_{2,\M'_t}^2.\nonumber
\end{eqnarray}
Now we have the bound of $\|\nabla u_t\|_{2,\M'_t}$ by combining \eqref{eq:error-grad-stable}, 
\eqref{eq:error-dl2-1-stable}, and \eqref{eq:error-dl2-3-stable}
\begin{align}
  \label{eq:error-grad-1-stable}
  &\|\nabla u_t\|_{2,\M'_t}^2\le \frac{C}{t}\int_{\M'_t}\int_{\M'_t}R_t(\bx,\by)(u_t(\bx)-u_t(\by))^2\mathd\bx\mathd\by\\
&+
\frac{C}{t}\int_{\M'_t}|u_t(\bx)|^2\left(\int_{\V_t}R_t(\bx,\by)\mathd\by\right)\mathd\bx
+Ct^2\left\|\nabla r(\bx)\right\|_{2,\M'_t}^2
+Ct\left\|r(\bx)\right\|_{2,\M'_t}^2.\nonumber
\end{align}

Then the bound of $\|\nabla u_t\|_{2,\M'_t}$ can be obtained also from \eqref{eq:error-grad-1-stable}
\begin{eqnarray}
&&  \|\nabla u_t\|_{2,\M'_t}^2\nonumber\\
&\le&\frac{C}{t}\int_{\M'_t}\int_{\M'_t}R_t(\bx,\by)(u_t(\bx)-u_t(\by))^2\mathd\bx\mathd\by+Ct\left\|r(\bx)\right\|_{2,\M'_t}^2\nonumber\\
&&+
\frac{C}{t}\int_{\M'_t}|u_t(\bx)|^2\left(\int_{\V_t}R_t(\bx,\by)\mathd\by\right)\mathd\bx+Ct^2\left\|\nabla r(\bx)\right\|_{2,\M'_t}^2
\nonumber\\
&\le&  \left| \frac{C}{t}\int_{\M'_t}u_t(\bx)\int_{\M}R_t(\bx,\by)(u_t(\bx)-u_t(\by))\mathd\by\mathd\bx\right|
\nonumber\\
&&+Ct^2\left\|\nabla r(\bx)\right\|_{2,\M'_t}^2
+Ct\left\|r(\bx)\right\|_{2,\M'_t}^2\nonumber\\
&\le&\|u_t\|_{2,\M'_t}\|r\|_{2,\M'_t}+ Ct^2\left\|\nabla r(\bx)\right\|_{2,\M'_t}^2
+Ct\left\|r(\bx)\right\|_{2,\M'_t}^2\nonumber\\
&\le&C\|r\|_{L^2(\M'_t)}^2+ Ct^2\left\|\nabla r(\bx)\right\|_{2,\M'_t}^2.\nonumber
\end{eqnarray}
Then we have
\begin{eqnarray}
  \label{eq:error-dl2-stable}
  \|\nabla u_t\|_{2, \M'_t}\le C\|r\|_{L^2(\M'_t)}+ Ct\left\|\nabla r(\bx)\right\|_{2,\M'_t}.
\end{eqnarray}
The proof is completed by putting \eqref{eq:error-l2-stable} and \eqref{eq:error-dl2-stable} together.
\end{proof}


\section{Convergence analysis}
\label{sec:converge}

The main purpose of this section is to prove that the solution of \eqref{eq:dis} converges to the solution of the original 
Poisson equation \eqref{eq:dirichlet}, i.e. Theorem \ref{thm:converge} in Section \ref{sec:PIM}. To prove this theorem, we split it 
to two parts. First, we prove that the solution of the integral equation \eqref{eq:integral} converges to the solution of the Poisson equation 
\eqref{eq:dirichlet}, which is given in Theorem \ref{thm:converge-integral}. Then we prove Theorem \ref{thm:converge-dis} to show that 
the solution of \eqref{eq:dis} converges to the solution of \eqref{eq:integral}.

\subsection{Integral approximation of Poisson equation}

To prove the convergence of the integral equation \eqref{eq:integral}, we need following theorem about the consistency
which is proved in \cite{SS-rate}.
\begin{theorem} Let $u(\bx)$ be the solution of the problem~\eqref{eq:dirichlet}.
 Let $u\in H^3(\M)$ and
 \begin{eqnarray*}
   r(\bx)=\frac{1}{t}\int_{\M}R_t(\bx,\by)(u(\bx)-u(\by))\mathd\by-\int_{\M}\bar{R}_t(\bx,\by)f(\by)\mathd\by.
 \end{eqnarray*}
There exists constants
$C, T_0$ depending only on $\M$ and $\p\M$, so that for any $t\le T_0$,
\begin{eqnarray}
\label{eq:integral_error_l2}
\left\|r(\bx)\right\|_{L^2(\M'_t)}&\le& Ct^{1/2}\|u\|_{H^3(\mathcal{M})},\\
\label{eq:integral_error_h1}
\left\|\nabla r(\bx)\right\|_{L^2(\M'_t)}
&\leq& C\|u\|_{H^3(\mathcal{M})}.
\end{eqnarray}
\label{thm:integral_error}
\end{theorem}
Using the consistency result, Theorem \ref{thm:integral_error} and the stability results presented in Section 
\ref{sec:stability}, we can get following theorem which shows the convergence of the integral equation \eqref{eq:integral}.
\begin{theorem}
  \label{thm:converge-integral}
Let $u(\bx)$ be solution of \eqref{eq:dirichlet} and $u_t(\bx)$ be solution of \eqref{eq:integral}. There exists $C>0$ only depends on $\M$ and $\p\M$, such that
\begin{eqnarray*}
  \|u-u_t\|_{H^1(\M'_t)}\le Ct^{1/4}\|f\|_{H^1(\M)}
\end{eqnarray*}
\end{theorem}
\begin{proof} 
Let $e_t(\bx)=u(\bx)-u_t(\bx)$, first of all, we have
\begin{align}
  \label{eq:error-1}
&  \frac{1}{t}\int_{\M'_t}e_t(\bx)\int_{\M}R_t(\bx,\by)(e_t(\bx)-e_t(\by))\mathd\by\mathd\bx\\
=& \frac{1}{t}\int_{\M'_t}e_t(\bx)\int_{\M'_t}R_t(\bx,\by)(e_t(\bx)-e_t(\by))\mathd\by\mathd\bx+
 \frac{1}{t}\int_{\M'_t}e_t(\bx)\int_{\V_t}R_t(\bx,\by)(e_t(\bx)-e_t(\by))\mathd\by\mathd\bx\nonumber\\
=& \frac{1}{2t}\int_{\M'_t}\int_{\M'_t}R_t(\bx,\by)(e_t(\bx)-e_t(\by))^2\mathd\bx\mathd\by+\frac{1}{t}\int_{\M'_t}e_t(\bx)\int_{\V_t}R_t(\bx,\by)(e_t(\bx)-e_t(\by))\mathd\by\mathd\bx.
\nonumber
\end{align}
The second term can be calculated as
\begin{eqnarray}
  \label{eq:error-2}
&&  \frac{1}{t}\int_{\M'_t}e_t(\bx)\int_{\V_t}R_t(\bx,\by)(e_t(\bx)-e_t(\by))\mathd\by\mathd\bx\\
&=&\frac{1}{t}\int_{\M'_t}|e_t(\bx)|^2\left(\int_{\V_t}R_t(\bx,\by)\mathd\by\right)\mathd\bx-\frac{1}{t}\int_{\M'_t}e_t(\bx)\left(\int_{\V_t}R_t(\bx,\by)u(\by)\mathd\by\right)\mathd\bx.
\nonumber
\end{eqnarray}
Here we use the definition of $e_t$ and the volume constraint condition $u_t(\bx)=0,\;\bx\in \V_t$ to get that $e_t(\bx)=u(\bx),\;\bx\in \V_t$. 

The first term is positive which is good for us. We only need to bound the second term of \eqref{eq:error-2} to show that it can be controlled by 
the first term. First, the second term can be bounded as following
\begin{eqnarray}
  \label{eq:error-3}\quad\quad
&&  \frac{1}{t}\left|\int_{\M'_t}e_t(\bx)\left(\int_{\V_t}R_t(\bx,\by)u(\by)\mathd\by\right)\mathd\bx\right|\\
&\le& \frac{1}{t}\int_{\M'_t}|e_t(\bx)|\left(\int_{\V_t}R_t(\bx,\by)\mathd\by\right)^{1/2}\left(\int_{\V_t}R_t(\bx,\by)|u(\by)|^2\mathd\by\right)^{1/2}\mathd\bx\nonumber\\
&\le&\frac{1}{t}\left(\int_{\M'_t}\frac{1}{2}|e_t(\bx)|^2\left(\int_{\V_t}R_t(\bx,\by)\mathd\by\right)\mathd\bx+
2\int_{\M'_t}\left(\int_{\V_t}R_t(\bx,\by)|u(\by)|^2\mathd\by\right)\mathd\bx\right)\nonumber\\
&\le& \frac{1}{2t}\int_{\M'_t}|e_t(\bx)|^2\left(\int_{\V_t}R_t(\bx,\by)\mathd\by\right)\mathd\bx+
\frac{2}{t}\int_{\V_t}|u(\by)|^2\left(\int_{\M'_t} R_t(\bx,\by)\mathd \bx\right)\mathd\by\nonumber\\
&\le& \frac{1}{2t}\int_{\M'_t}|e_t(\bx)|^2\left(\int_{\V_t}R_t(\bx,\by)\mathd\by\right)\mathd\bx+
\frac{C}{t}\int_{\V_t}|u(\by)|^2\mathd\by\nonumber\\
&\le& \frac{1}{2t}\int_{\M'_t}|e_t(\bx)|^2\left(\int_{\V_t}R_t(\bx,\by)\mathd\by\right)\mathd\bx+
C\sqrt{t}\|f\|^2_{H^1(\M)}.\nonumber
\end{eqnarray}
Here we use Lemma \ref{lem:u-boundary} in Appendix A to get the last inequality.
By substituting \eqref{eq:error-3}, \eqref{eq:error-2} in \eqref{eq:error-1}, we get
\begin{eqnarray}
  \label{eq:error-est}
&& \left| \frac{1}{t}\int_{\M'_t}e_t(\bx)\int_{\M}R_t(\bx,\by)(e_t(\bx)-e_t(\by))\mathd\by\mathd\bx\right|\\
&\ge&  \frac{1}{2t}\int_{\M'_t}\int_{\M'_t}R_t(\bx,\by)(e_t(\bx)-e_t(\by))^2\mathd\bx\mathd\by\nonumber\\
&&+
\frac{1}{2t}\int_{\M'_t}|e_t(\bx)|^2\left(\int_{\V_t}R_t(\bx,\by)\mathd\by\right)\mathd\bx-C\|f\|_{H^1(\M)}^2\sqrt{t}\nonumber.
\end{eqnarray}
This is the key estimate we used to get convergence. 

Notice that $e_t(\bx)$ satisfying an integral equation,
\begin{eqnarray}
\label{eq:integral-error}
  \frac{1}{t}\int_{\M}R_t(\bx,\by)(e_t(\bx)-e_t(\by))\mathd\by=r(\bx),\quad \forall \bx\in \M'_t,
\end{eqnarray}
where $r(\bx)=\frac{1}{t}\int_{\M}R_t(\bx,\by)(u(\bx)-u(\by))\mathd\by-\int_{\M}\bar{R}_t(\bx,\by)f(\by)\mathd\by$.

From Theorem \ref{thm:integral_error}, we know that 
\begin{eqnarray}
  \label{eq:residual-l2}
  &&\left\|r(\bx)\right\|_{L^2(\M'_t)}\le Ct^{1/2}\|u\|_{H^3(\mathcal{M})}\le C\sqrt{t}\|f\|_{H^1(\mathcal{M})},\\
 \label{eq:residual-dl2}
&&\left\|\nabla r(\bx)\right\|_{L^2(\M'_t)}\le C\|u\|_{H^3(\mathcal{M})}\le C\|f\|_{H^1(\mathcal{M})}.
\end{eqnarray}
 Now we can get $L^2$ estimate of $e_t$.  
Using Corollary \ref{cor:l2-inner}, we have 
\begin{align}
  \label{eq:est-l2}
  \|e_t\|_{2,\M'_t}^2&\le\frac{C}{t}\int_{\M'_t}\int_{\M'_t}R_t(\bx,\by)(e_t(\bx)-e_t(\by))^2\mathd\bx\mathd\by
\\
&+\frac{C}{t}\int_{\M'_t}|e_t(\bx)|^2\left(\int_{\V_t}R_t(\bx,\by)\mathd\by\right)\mathd\bx\nonumber\\
\text{(from \eqref{eq:error-est})}\quad
&\le  \left| \frac{C}{t}\int_{\M'_t}e_t(\bx)\int_{\M}R_t(\bx,\by)(e_t(\bx)-e_t(\by))\mathd\by\mathd\bx\right|+
C\|f\|_{H^1(\M)}^2\sqrt{t}\nonumber\\
\text{(from \eqref{eq:integral-error})}\quad&\le C\|e_t\|_{2,\M'_t}\|r\|_{2,\M'_t}+C\|f\|_{H^1(\M)}^2\sqrt{t}\nonumber\\
\text{(from \eqref{eq:residual-l2})}\quad&\le C\|f\|_{H^1(\M)}\|e_t\|_{2,\M'_t}\sqrt{t}+C\|f\|_{H^1(\M)}^2\sqrt{t}.\nonumber
\end{align}
This gives that
\begin{eqnarray}
  \label{eq:error-l2}
  \|e_t\|_{2, \M'_t}\le Ct^{1/4}\|f\|_{H^1(\M)}.
\end{eqnarray}

Next, we turn to estimate the $L^2$ norm of $\nabla e_t$ in $\M'_t$. 
Using the integral equation \eqref{eq:integral-error}, $e_t$ has following expression
\begin{eqnarray}
  \label{eq:error-exp}
  &&e_t(\bx)=\frac{1}{w_t(\bx)}\int_{\M}R_t(\bx,\by)e_t(\by)\mathd\by+\frac{t}{w_t(\bx)}r(\bx)\\
&=&\frac{1}{w_t(\bx)}\int_{\M'_t}R_t(\bx,\by)e_t(\by)\mathd\by+\frac{1}{w_t(\bx)}\int_{\V_t}R_t(\bx,\by)u(\by)\mathd\by+\frac{t}{w_t(\bx)}r(\bx).
\nonumber
\end{eqnarray}
Then $\|\nabla e_t\|_{2,\M'_t}^2$ can be bounded as following
\begin{eqnarray}
  \label{eq:error-grad}
 && \|\nabla e_t\|_{2,\M'_t}^2\le C\left\|\nabla\left(\frac{1}{w_t(\bx)}\int_{\M'_t}R_t(\bx,\by)e_t(\by)\mathd\by\right)\right\|_{2,\M'_t}^2\\
&&+
C\left\|\nabla\left(\frac{1}{w_t(\bx)}\int_{\V_t}R_t(\bx,\by)u(\by)\mathd\by\right)\right\|_{2,\M'_t}^2
+Ct^2\left\|\nabla\left(\frac{r(\bx)}{w_t(\bx)}\right)\right\|_{2,\M'_t}^2.\nonumber
\end{eqnarray}
Corollary \ref{cor:coercivity-inner} gives a bound the first term of \eqref{eq:error-grad}.
\begin{align}
  \label{eq:error-dl2-1}
&  \left\|\nabla\left(\frac{1}{w_t(\bx)}\int_{\M'_t}R_t(\bx,\by)e_t(\by)\mathd\by\right)\right\|_{2,\M'_t}^2\\
&\le 
\frac{C}{t}\int_{\M'_t}\int_{\M'_t}R_t(\bx,\by)(e_t(\bx)-e_t(\by))^2\mathd\by\mathd\bx
+\frac{C}{t}\int_{\M'_t}|e_t(\bx)|^2\left(\int_{\V_t}R_t(\bx,\by)\mathd\by\right)\mathd\bx.\nonumber
\end{align}

The second and third terms of \eqref{eq:error-grad} can be bounded by direct calculation.
\begin{eqnarray}
  \label{eq:error-dl2-3}
  \left\|\nabla\left(\frac{r(\bx)}{w_t(\bx)}\right)\right\|_{2,\M'_t}^2&\le& C\left\|\frac{\nabla r(\bx)}{w_t(\bx)}\right\|_{2,\M'_t}^2
+C\left\|\frac{ r(\bx)\nabla w_t(\bx)}{(w_t(\bx))^2}\right\|_{2,\M'_t}^2\\
&\le& C\left\|\nabla r(\bx)\right\|_{2,\M'_t}^2
+\frac{C}{t}\left\|r(\bx)\right\|_{2,\M'_t}^2\nonumber\\
&\le& C\|f\|_{H^1(\M)}^2,\nonumber
\end{eqnarray}
and
\begin{eqnarray}
  \label{eq:est-grad-1}
&&\left|  \nabla\left(\frac{1}{w_t(\bx)}\int_{\V_t}R_t(\bx,\by)u(\by)\mathd\by\right)\right|\\
&\le&\left|\frac{\nabla w_t(\bx)}{(w_t(\bx))^2}\int_{\V_t}R_t(\bx,\by)u(\by)\mathd\by\right|
+\left|\frac{1}{w_t(\bx)}\int_{\V_t}\nabla R_t(\bx,\by)u(\by)\mathd\by\right|\nonumber\\
&\le& C\|f\|_{H^1(\M)}\int_{\V_t}R_t(\bx,\by)\mathd\by+C\|f\|_{H^1(\M)}\int_{\V_t}\left|R'_t(\bx,\by)\right|\mathd\by.\nonumber
\end{eqnarray}
Then the second term of \eqref{eq:error-dl2-1} has following bound
\begin{eqnarray}
  \label{eq:error-dl2-2}
&& \quad\quad\quad \left\|\nabla\left(\frac{1}{w_t(\bx)}\int_{\V_t}R_t(\bx,\by)u(\by)\mathd\by\right)\right\|_{2,\M'_t}^2\\
&\le& C\|f\|_{H^1(\M)}^2 \int_{\M'_t}\left(\int_{\V_t}R_t(\bx,\by)\mathd\by\right)^2\mathd \bx
+C\|f\|_{H^1(\M)}^2 \int_{\M'_t}  \left(\int_{\V_t}\left|R'_t(\bx,\by)\right|\mathd\by\right)^2\mathd\bx\nonumber\\
&\le &C \|f\|_{H^1(\M)}^2\int_{\M'_t}\int_{\V_t}R_t(\bx,\by)\mathd\by\mathd \bx+C \|f\|_{H^1(\M)}^2\int_{\M'_t}  \int_{\V_t}\left|R'_t(\bx,\by)\right|\mathd\by\mathd\bx\nonumber\\
&\le&C\|f\|_{H^1(\M)}^2|\V_t|\le C\|f\|_{H^1(\M)}^2\sqrt{t}.\nonumber
\end{eqnarray}
Now we have the bound of $\|\nabla e_t\|_{2,\M'_t}$ by combining \eqref{eq:error-grad}, \eqref{eq:error-dl2-3}, \eqref{eq:error-dl2-1} and \eqref{eq:error-dl2-2}
\begin{eqnarray}
\label{eq:error-grad-1}
  \|\nabla e_t\|_{2,\M'_t}^2&\le& \frac{C}{t}\int_{\M'_t}\int_{\M'_t}R_t(\bx,\by)(e_t(\bx)-e_t(\by))^2\mathd\bx\mathd\by\\
&&+\frac{C}{t}\int_{\M'_t}|e_t(\bx)|^2\left(\int_{\V_t}R_t(\bx,\by)\mathd\by\right)\mathd\bx+
C\|f\|_{H^1(\M)}^2\sqrt{t}.\nonumber
\end{eqnarray}

Then the bound of $\|\nabla e_t\|_{2,\M'_t}$ can be obtained also from \eqref{eq:error-grad-1}
\begin{align}
  \label{eq:est-dl2}
  \|\nabla e_t\|_{2,\M'_t}^2\le&\frac{C}{t}\int_{\M'_t}\int_{\M'_t}R_t(\bx,\by)(e_t(\bx)-e_t(\by))^2\mathd\bx\mathd\by\\
&+
\frac{C}{t}\int_{\M'_t}|e_t(\bx)|^2\left(\int_{\V_t}R_t(\bx,\by)\mathd\by\right)\mathd\bx+
C\|f\|_{H^1(\M)}^2\sqrt{t}\nonumber\\
\text{(from \eqref{eq:error-est})}\quad
\le&  \left| \frac{C}{t}\int_{\M'_t}e_t(\bx)\int_{\M}R_t(\bx,\by)(e_t(\bx)-e_t(\by))\mathd\by\mathd\bx\right|+ C\|f\|_{H^1(\M)}^2\sqrt{t}\nonumber\\
\text{(from \eqref{eq:integral-error})}\quad\le&\|e_t\|_{2,\M'_t}\|r\|_{2,\M'_t}+ C\|f\|_{H^1(\M)}^2\sqrt{t}\nonumber\\
\text{(from \eqref{eq:residual-l2})}\quad\le&C\|f\|_{H^1(\M)}\|e_t\|_{2,\M'_t}\sqrt{t}+ C\|f\|_{H^1(\M)}^2\sqrt{t}\nonumber\\
\text{(from \eqref{eq:error-l2})}\quad\le&Ct^{3/4}\|f\|_{H^1(\M)}+ C\|f\|_{H^1(\M)}^2\sqrt{t}.\nonumber
\end{align}
Then we have
\begin{eqnarray}
  \label{eq:error-dl2}
  \|\nabla e_t\|_{2, \M'_t}\le Ct^{1/4}\|f\|_{H^1(\M)}.
\end{eqnarray}
The proof is completed by putting \eqref{eq:error-l2} and \eqref{eq:error-dl2} together.
\end{proof}

\subsection{Discretization of the integral equation}
\label{sec:discrete}
Suppose $\bfu=[u_1,\cdots,u_n]^t$ is the discrete solution which means that it solves \eqref{eq:dis}.
First, we interpolate the discrete solution from the point cloud $P=\{\bfp_1,\cdots,\bfp_n\}$ to the whole 
manifold $\M$. Fortunately, the discrete equation \eqref{eq:dis} gives a natural interpolation. 
\begin{equation}
  \label{eq:interpolation}
  u_{t,h}(\bx)=\left\{\begin{array}{cc}\D
\frac{1}{w_{t,h}(\bx)}\left(\sum_{\bfp_j\in\M}R_t(\bx,\bfp_j)u_jV_j
 +t\sum_{\bfp_j\in\M}\bar{R}_t(\bx,\bfp_j)f(\bfp_j)V_j\right),&\bx\in \M'_t,\\
0,&\bx\in \V_t.
\end{array}\right.
\end{equation}
Then, we have following theorem regarding the convergence from $u_{t,h}$ to $u_t$.
\begin{theorem}
Let $u_t(\bx)$ be the solution of the problem~~\eqref{eq:integral} and $\bfu$ be the solution of 
the problem~\eqref{eq:dis}. If 
$f\in C^1(\M)$
in both problems, then there exists constants
$C>0$ depending only on $\M$ and  $\p \M$ so that 
\begin{eqnarray}
\|u_{t,h} - u_{t}\|_{H^1(\M'_t)} &\leq& \frac{Ch}{t^{3/2}}\|f\|_{C^1(\M)},\nonumber
\end{eqnarray}
as long as $t$ and $\frac{h}{\sqrt{t}}$ are both small enough.
\label{thm:converge-dis}
\end{theorem}
To prove this theorem, we need the stability result, Theorem \ref{thm:stability}, and the 
consistency result which is given in Theorem \ref{thm:dis_error}.

To simplify the notations, we introduce some operators here.
\begin{eqnarray} \label{eq:Lt}
 L_tu(\bx)&=& \frac{1}{t}\int_{\M'_t}R_t(\bx,\by)(u(\bx)-u(\by))\mathd\by
+\frac{u(\bx)}{t}\int_{\V_t}R_t(\bx,\by)\mathd\by,\\
L'_tu(\bx)&=& \frac{1}{t}\int_{\M'_t}R_t(\bx,\by)(u(\bx)-u(\by))\mathd\by,
\label{eq:Lt-inner}
\\
  \label{eq:dis-operator}
 \quad\quad (L_{t,h}u)(\bx)&=&\frac{1}{t}\sum_{\bfp_j\in\M'_t}R_t(\bx,\bfp_j)(u(\bx)-u(\bfp_j))V_j
+\frac{u(\bx)}{t}\sum_{\bfp_j\in\V_t}R_t(\bfp_i,\bfp_j)V_j,\\
\label{eq:dis-operator-inner}
\quad\quad(L'_{t,h}u)(\bx)&=&\frac{1}{t}\sum_{\bfp_j\in\M'_t}R_t(\bx,\bfp_j)(u(\bx)-u(\bfp_j))V_j.
\end{eqnarray}
It is easy to check that $u_{t,h}$ satisfies following equation if $\bfu=[u_1,\cdots,u_n]^t$ solves \eqref{eq:dis},
\begin{eqnarray}
  \label{eq:dis_interp}
  L_{t,h}u_{t,h}=\sum_{\bfp_j\in\M}\bar{R}_t(\bx,\bfp_j)f(\bfp_j)V_j
\end{eqnarray}
Now, we can state the consistency result as following.
\begin{theorem}
Let $u_t(\bx)$ be the solution of the problem~~\eqref{eq:integral} and $\bfu$ be the solution of 
the problem~\eqref{eq:dis}. If 
$f\in C^1(\M)$ ,
in both problems, then there exists constants
$C>0$ depending only on $\M$ and  $\p \M$ so that 
\begin{eqnarray}
\|L_{t} \left(u_{t,h} - u_{t}\right)\|_{L^2(\M'_t)} &\leq& \frac{Ch}{t^{3/2}}\|f\|_{C^1(\M)},\\
\label{eqn:dis_error_l2}
\|\nabla L_{t} \left(u_{t,h} - u_{t}\right)\|_{L^2(\M'_t)} &\leq& \frac{Ch}{t^{2}}\|f\|_{C^1(\M)}.
\label{eqn:dis_error_dl2}
\end{eqnarray}
as long as $t$ and $\frac{h}{\sqrt{t}}$ are small enough.
\label{thm:dis_error}
\end{theorem}

\section{Convergence of the eigenvalue problem}
\label{sec:eigen}

In this section, we investigate the convergence of the eigenvalue problem~\eqref{eq:dis-eigen} to the 
eigenvalue problem~\eqref{eq:eigen}. First, we introduce some operators. 

Denote the operator $T: L^2(\mathcal{M})\rightarrow H^2(\mathcal{M})$ to be the solution operator of
the following problem 
\begin{eqnarray}
  \left\{\begin{array}{rl}
      \Delta_\M u(\bx)=f(\bx),&\bx\in \mathcal{M},\\
      u(\bx)=0,& \bx\in \p \mathcal{M}.
\end{array}\right.\nonumber
\end{eqnarray}
where $\bn$ is the out normal vector of $\mathcal{M}$.

Denote $T_t:L^2(\mathcal{M})\rightarrow L^2(\mathcal{M})$ to be the solution operator of the following
problem 
\begin{eqnarray} 
\left \{
\begin{array}{rlll}
 \D   -\frac{1}{t}\int_{\M}R_t(\bx,\by)(u(\bx)-u(\by))\mathd\by&=&\D \int_{\mathcal{M}} \bar{R}_t(\bx, \by) f(\by) \mathd \by, 
\quad & \bx \in \mathcal{M}'_t \\
    u(\bx) & = & 0,\quad & \bx \in \mathcal{V}_t \\
\end{array}
\right.\nonumber
\end{eqnarray}

The last solution operator is $T_{t,h}:C(\mathcal{M})\rightarrow C(\mathcal{M})$ which is defined as follows. 
\begin{eqnarray}
T_{t,h}(f)(\bx)=\left\{\begin{array}{cc} \D\frac{1}{w_{t,h}(\bx)}\left(\sum_{\bfp_j\in \M}R_t(\bx,\bfp_j)u_jV_j
-t\sum_{\bfp_j\in \M}\bar{R}_t(\bx,\bfp_j)f(\bfp_j)V_j\right),& \bx\in \M'_t,\\
0,&\bx\in \V_t,
\end{array}\right.\nonumber
\end{eqnarray}
where $w_{t,h}(\bx)=\sum_{\bfp_j\in \M}R_t(\bx,\bfp_j)V_j$ and $\bfu=(u_1, \cdots, u_n)^t$ with $u_j=0, \; \bfp_j\in \V_t$ solves
the following linear system 
\begin{eqnarray}
 -\frac{1}{t}\sum_{\bfp_j\in \M}R_t(\bfp_i,\bfp_j)(u_i-u_j)V_j=\sum_{\bfp_j\in \M}\bar{R}_t(\bfp_i,\bfp_j)f(\bfp_j)V_j.\nonumber
\end{eqnarray}
We know that $T, T_t$ and $T_{t,h}$ have following properties. 
\begin{proposition}
\label{prop:eigen}
For any $t>0$, $h>0$,
  \begin{itemize}
  \item[1. ] $T, T_t$ are compact operators on $H^1(\M)$ into $H^1(\M)$; $T_t, T_{t,h}$ are compact operators on $C^1(\M)$ into $C^1(\M)$.
\item[2. ] All eigenvalues of $T, T_t, T_{t,h}$ are real numbers. All generalized eigenvectors of $T, T_t, T_{t,h}$ are eigenvectors.
  \end{itemize}
\end{proposition}
\begin{proof}
  The proof of (1) is straightforward. First, it is well known that $T$ is compact operator. $T_{t,h}$ is actually finite dimensional operator, so
it is also compact. To show the compactness of $T_t$, we need the following formula,
\begin{eqnarray}
  T_t u = \frac{1}{w_t(\bx)}\int_\M R_t(\bx,\by)T_tu(\by)\mathd \by+ \frac{t}{w_t(\bx)}\int_\M \bar{R}_t(\bx,\by)u(\by)\mathd \by, \quad \forall u\in H^1(\M).
\nonumber
\end{eqnarray}
Using the assumption that $R\in C^2$, direct calculation would gives that
 that $T_tu\in C^2$. This would imply the compactness of $T_t$ both in $H^1$ and $C^1$.

For the operator $T$, the conclusion (2) is well known. The proof of $T_t$ and $T_{t,h}$ are very similar, so here we only present the
proof for $T_t$.

Let $\lambda$ be an eigenvalue of $T_t$ and $u$ is corresponding eigenfunction, then
\begin{eqnarray*}
  L_t T_t u =\lambda L_t u
\end{eqnarray*}
which implies that
\begin{eqnarray*}
  \lambda = \frac{\int_\M \int _\M \bar{R}_t(\bx, \by) u^*(\bx)u(\by) \mathd \bx\mathd \by}{\int_\M u^*(\bx)(L_t u)(\bx)\mathd \bx }
\end{eqnarray*}
where $u^*$ is the complex conjugate of $u$.

Using the symmetry of $L_t$ and $\bar{R}(\bx,\by)$, it is easy to show that $\lambda\in \mathbb{R}$.

Let $u$ be a generalized eigenfunction of $T_t$ with multiplicity $m>1$ associate with eigenvalue $\lambda$. Let $v=(T_t-\lambda)^{m-1}u$,
$w=(T_t-\lambda)^{m-2}u$, then $v$
is an eigenfunction of $T_t$ and
\begin{eqnarray*}
  T_t v =\lambda v,\quad (T_t-\lambda)w=v
\end{eqnarray*}
and $v(\bx)=0,\; w(\bx)=0,\; \bx\in \V_t$.

By applying $L_t$ on both sides of above two equations, we have
\begin{eqnarray*}
  \lambda L_t v &=& L_t (T_tv) = \int_\M \bar{R}_t(\bx,\by)v(\by)\mathd \by=\int_{\M'_t} \bar{R}_t(\bx,\by)v(\by)\mathd \by,\quad \bx\in \M'_t\\
L_t v &=& L_t(T_tw)-\lambda L_t w = \int_{\M'_t} \bar{R}_t(\bx,\by)w(\by)\mathd \by-\lambda L_t w,\quad \bx\in \M'_t
\end{eqnarray*}
Using above two equations and the fact that $L_t$ is symmetric, we get
\begin{eqnarray*}
  0&=& \left<w,\lambda L_t v - \int_{\M'_t} \bar{R}_t(\bx,\by)v(\by)\mathd \by\right>_{\M'_t}\nonumber\\
&=& \left<\lambda L_t w - \int_{\M'_t} \bar{R}_t(\bx,\by)w(\by)\mathd \by, v\right>_{\M'_t}\nonumber\\
&=& \left<L_t v, v\right>_{\M'_t}\ge C\left\|v\right\|_2^2
\end{eqnarray*}
which implies that $(T_t-\lambda)^{m-1}u=v=0$. This proves that $u$ is a generalized eigenfunction of $T_t$ with multiplicity $m-1$.
Repeating above argument, we can show that $u$ is actually an eigenfunction of $T_t$.
\end{proof}

Theorem \ref{thm:converge-integral} actually gives that $T_t$ converges to $T$ in $H^1$ norm.
\begin{theorem}
\label{thm:converge_h1} Under the assumptions in Section \ref{sec:PIM}, for $t$ small enough, 
there exists a constant $C>0$ such that
  \begin{eqnarray}
    \|T-T_t\|_{H^1}\le Ct^{1/4}.\nonumber
  \end{eqnarray}
\end{theorem}
Using the arguments in \cite{SS-rate} and Theorem \ref{thm:converge-dis}, we can get that 
$T_{t,h}$ converges to $T_t$ in $C^1$ norm.
\begin{theorem}
\label{thm:converge_c1}
Under the assumptions in Section \ref{sec:PIM}, for $t,h$ small enough, 
there exists a constant $C>0$ such that
  \begin{eqnarray}
    \|T_{t,h}-T_t\|_{C^1}\le \frac{Ch}{t^{d/4+2}}.\nonumber
  \end{eqnarray}
\end{theorem}
And we also have the bound of $T_t$ and $T_{t,h}$ following the arguments in \cite{SS-rate}.
\begin{theorem}
\label{thm:bound} Under the assumptions in Section \ref{sec:PIM}, for $t,h$ small enough, 
there exists a constant $C$ independent on $t$ and $h$, such that
  \begin{eqnarray}
    \|T_t\|_{H^1}\le C,\quad \|T_{t,h}\|_{\infty}\le Ct^{-d/4}, \quad
\|T_{t,h}\|_{C^1}\le Ct^{-(d+2)/4}.\nonumber
  \end{eqnarray}
\end{theorem}

. 
Before state the main theorem of the spectral convergence, we need to introduce some notations. 
Let $X$ be a complex Banach space and $L:X\rightarrow X$ be a compact linear operator. $\rho(L)$ is the resolvent set of $L$ which  is given by
$z\in \mathbb{C}$
such that $z-L$ is bijective. The spectrum of $L$ is $\sigma(L)=\mathbb{C}\backslash\rho(L)$. 
If $\lambda$ is a nonzero eigenvalue of $L$, the ascent multiplicity $\al$ of $\lambda-L$ is the smallest integer such that $\ker(\lambda-L)^\al=\ker(\lambda-L)^{\al+1}$.

Given a closed smooth curve $\Gamma\subset \rho(L)$ which encloses the eigenvalue $\lambda$ and no other elements of $\sigma(L)$, the Riesz spectral projection associated with
$\lambda$ is defined by
\begin{eqnarray}
  E(\lambda, L)=\frac{1}{2\pi i}\int_{\Gamma} (z-L)^{-1}\mathd z,\nonumber
\end{eqnarray}
where $i=\sqrt{-1}$ is the unit imaginary


Now we are ready to state the main theorem about the convergence of the eigenvalue problem. And its proof
can be given from Theorems \ref{thm:converge_h1}, \ref{thm:converge_c1} and \ref{thm:bound} following same 
arguments as those in \cite{SS-rate}.
\begin{theorem}
Under the assumptions in Section \ref{sec:PIM},
let $\lambda_i$ be the $i$th smallest eigenvalue of $T$ counting multiplicity, and $\lambda_{i}^{t,h}$ be the $i$th smallest eigenvalue
of $T_{t,h}$ counting multiplicity, then there exists a constant $C$ such that
$$|\lambda_i^{t,h} - \lambda_i|\le C  \left(t^{1/4}+\frac{h}{t^{d/4+3}}\right) , $$
and there exist another constant $C$ such that, for any $\phi\in E(\lambda_i,T)X$ and $X=H^1(\M)$,
$$\|\phi-E(\sigma_{i}^{t,h},T_{t,h})\phi\|_{H^1(\M)} \le C \left(t^{1/4}+\frac{h}{t^{d/4+2}}\right) .$$
where $\sigma_i^{t,h}=\{\lambda_j^{t,h}\in \sigma(T_{t,h}): j\in I_i\}$ and $I_i=\{j\in \mathbb{N}: \lambda_j=\lambda_i\}$.
\label{thm:eigen}
\end{theorem}

The convergence result, Thorem \ref{thm:eigen_converge},
 follows easily from the above theorem and Proposition~\ref{prop:eigen}.

\section{Numerical results}
\label{sec:num}
In this section, we present several numerical results to show the convergence of the Point Integral method 
with volume constraint, PIM\_VC for short, from
point clouds. 

The numerical experiments were carried out in unit disk. 
We discretize unit disk with 684, 2610, 10191 and 40269 points respectively and check the convergence of the 
point integral method with volume constraint.
In the experiments, the volume weight vector $\bf V$ is estimated using the method 
proposed in~\cite{LuoSW09}. First, we locally approximate the tangent space at each point and then 
project the nearby points onto the tangent space over which a Delaunay triangulation is computed in 
the tangent space. The volume weight is estimated as the volume of the Voronoi cell of that point.




Table \ref{tbl:unit_disk_cos} gives the $l^2$ error of different methods with 684, 2610, 10191 and 40269 points.
The exact solution is $\cos 2\pi\sqrt{x^2+y^2}$. 
PIM\_Robin is the Point Integral method and using Robin boundary to 
approximate the Dirichlet boundary condition, i.e. solving the integral equation \eqref{eq:integral-pim-robin}
and here $\beta$ is chosen to be $10^{-4}$.
PIM\_VC is the Point Integral method and using volume constraint to enforce the Dirichlet boundary condition.
These two methods both converge. The rates of convergence are very close and the error of PIM\_VC is a little 
larger than the error of PIM\_Robin.
\begin{table}[!ht]
\begin{center}
\begin{tabular}{| c|| c | c | c | c |}
\hline
$|P|$ & 684    &  2610 & 10191 & 40269 \\
\hline
PIM\_Robin      &  0.1500 &   0.0428  &  0.0140  &  0.0052\\
\hline
PIM\_VC      &  0.3046 &  0.0747  &  0.0201   &  0.0067\\
\hline
\end{tabular}
\end{center}
\caption{$l^2$ error with different number of points. FEM: Finite Element method; 
PIM\_Robin: Point Integral method with Robin boundary; 
PIM\_VC: Point Integral method with volume constraint. 
The exact solution is $\cos 2\pi \sqrt{x^2+y^2}$.
\label{tbl:unit_disk_cos}}
\end{table}

Fig. \ref{fig:unit_disk_eigen} shows the result of the eigenvalues of Laplace-Beltrami operator with 
Dirichlet boundary in unit disk. 
Clearly, the eigenvalues also converge and the larger eigenvalues have larger errors which verify the 
theoretical result, Theorem \ref{thm:eigen_converge}. 
\begin{figure}[h]
\begin{center}
\begin{tabular}{c}
\includegraphics[width=0.7\textwidth]{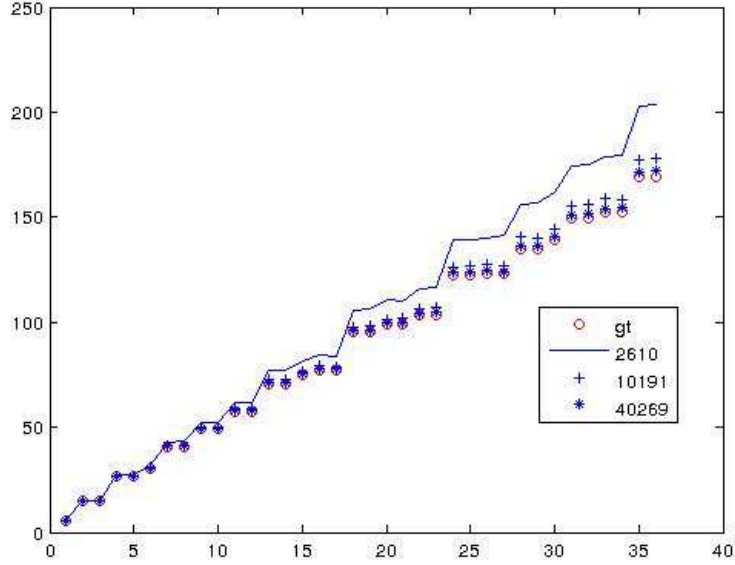} 
\end{tabular}
\end{center}
\vspace{-4mm}
\caption{Eigenvalue given by Point Integral method with volume constraint in unit disk.}
\label{fig:unit_disk_eigen}
\end{figure}

Above numerical results in unit disk are just toy examples to demonstrate the convergence of the Point Integral method with
volume constraint. However, our method applies in any point clouds which sample smooth manifolds. 
Fig. \ref{fig:various_eigen} shows the first two eigenfunctions on two complicated surfaces 
(left hand and head of Max Plank). 
\begin{figure}[h]
\begin{center}
\begin{tabular}{cccc}
\includegraphics[width=0.24\textwidth]{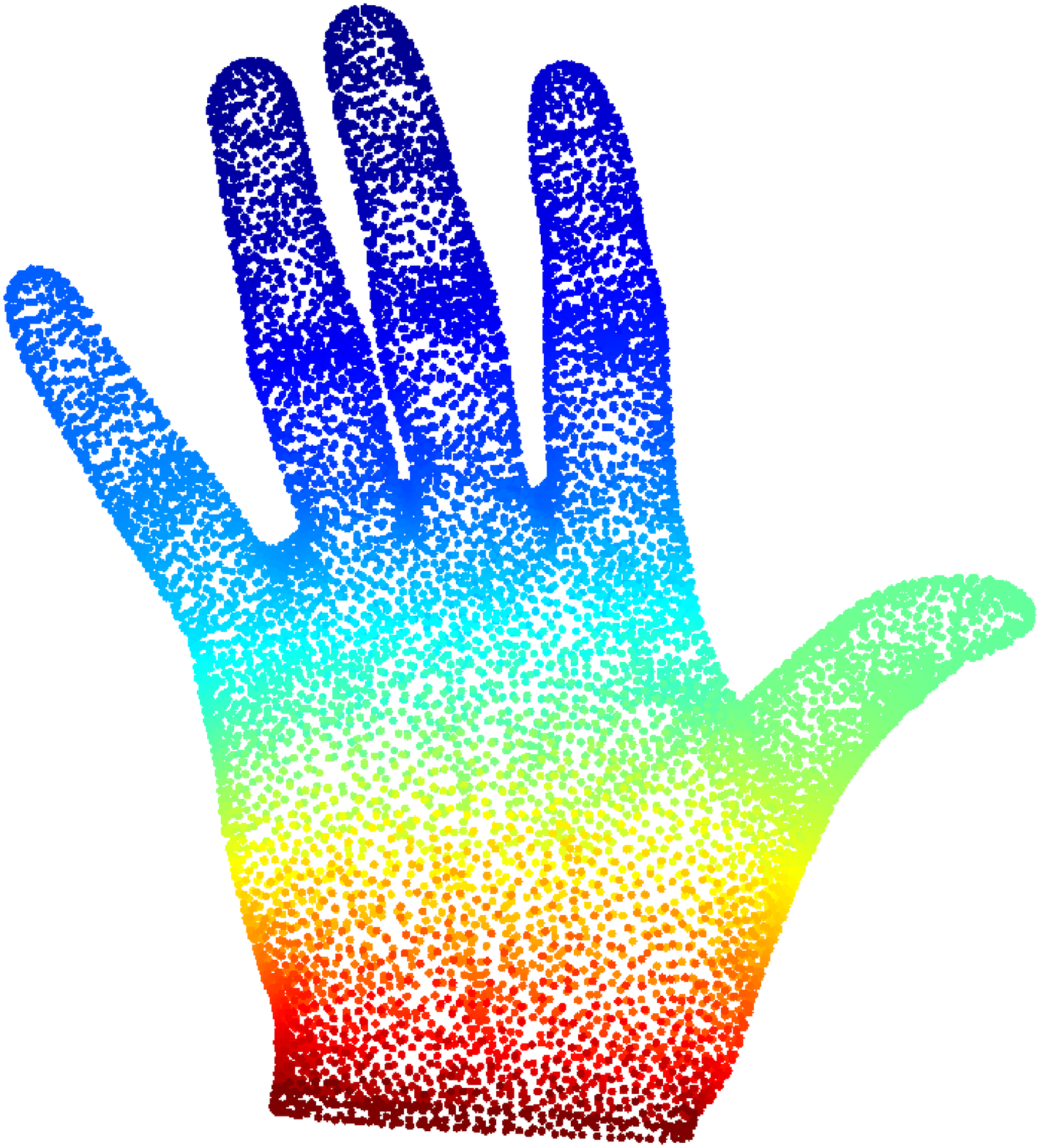} &
\includegraphics[width=0.25\textwidth]{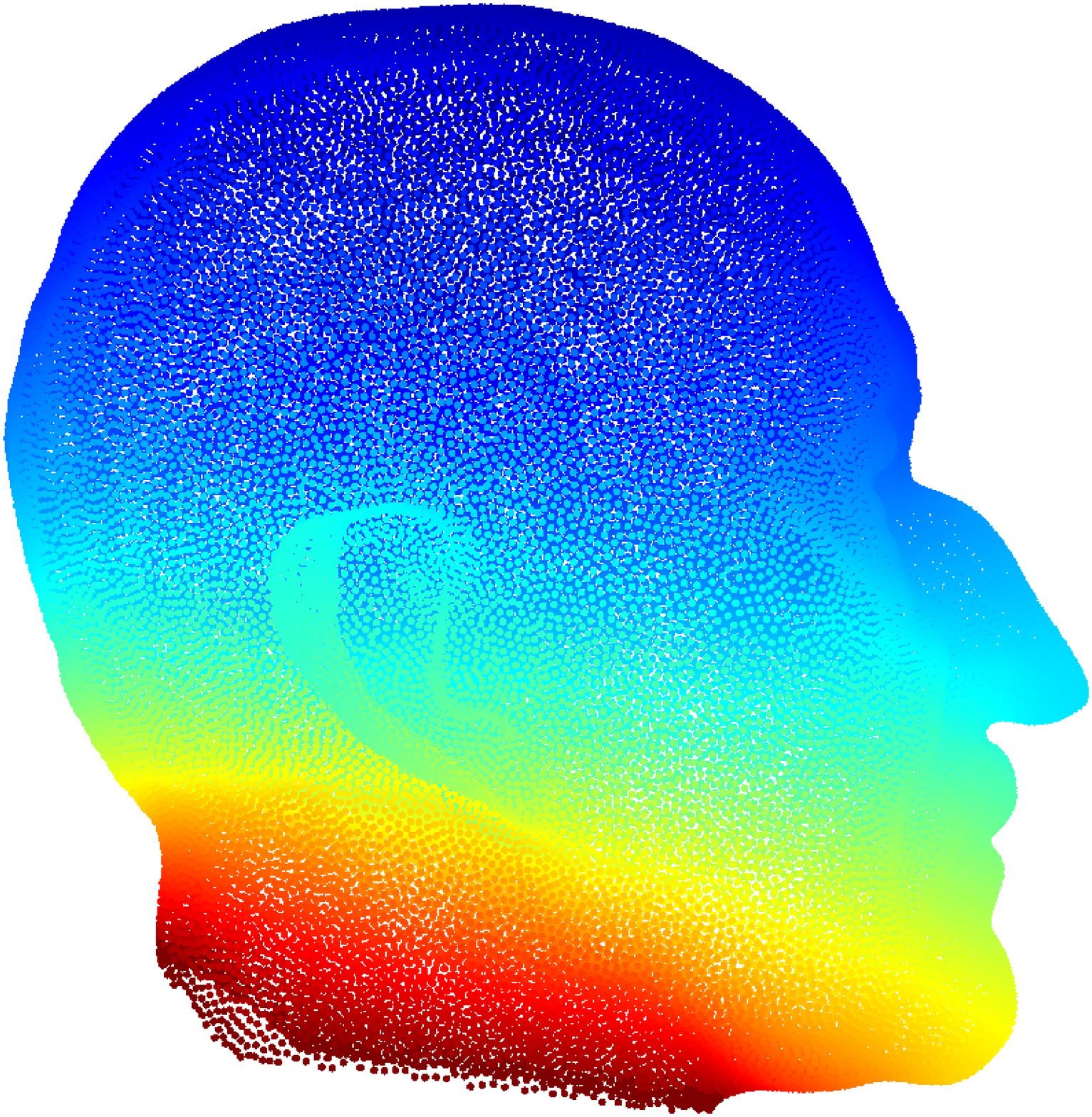} \\
\includegraphics[width=0.24\textwidth]{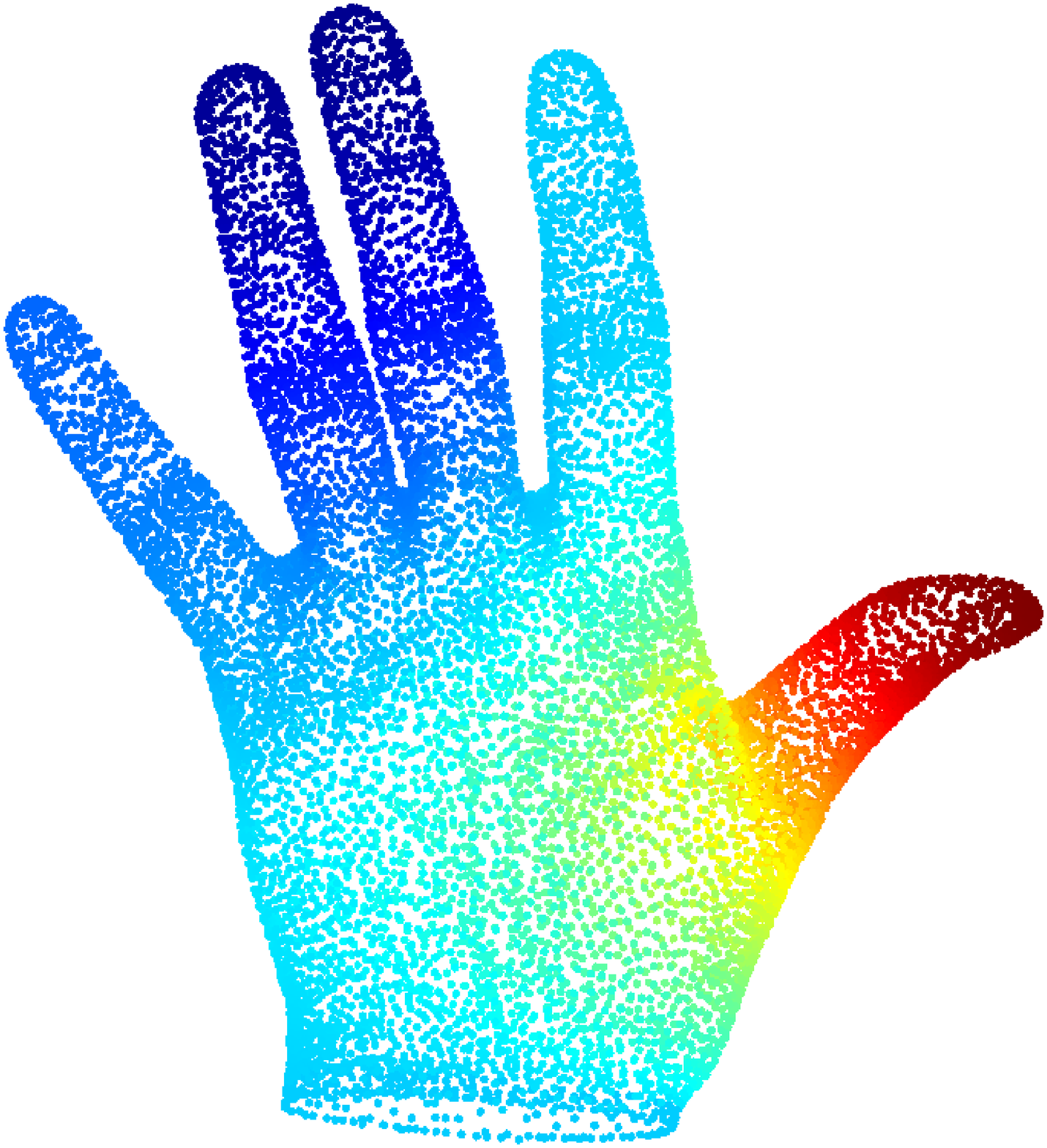} &
\includegraphics[width=0.24\textwidth]{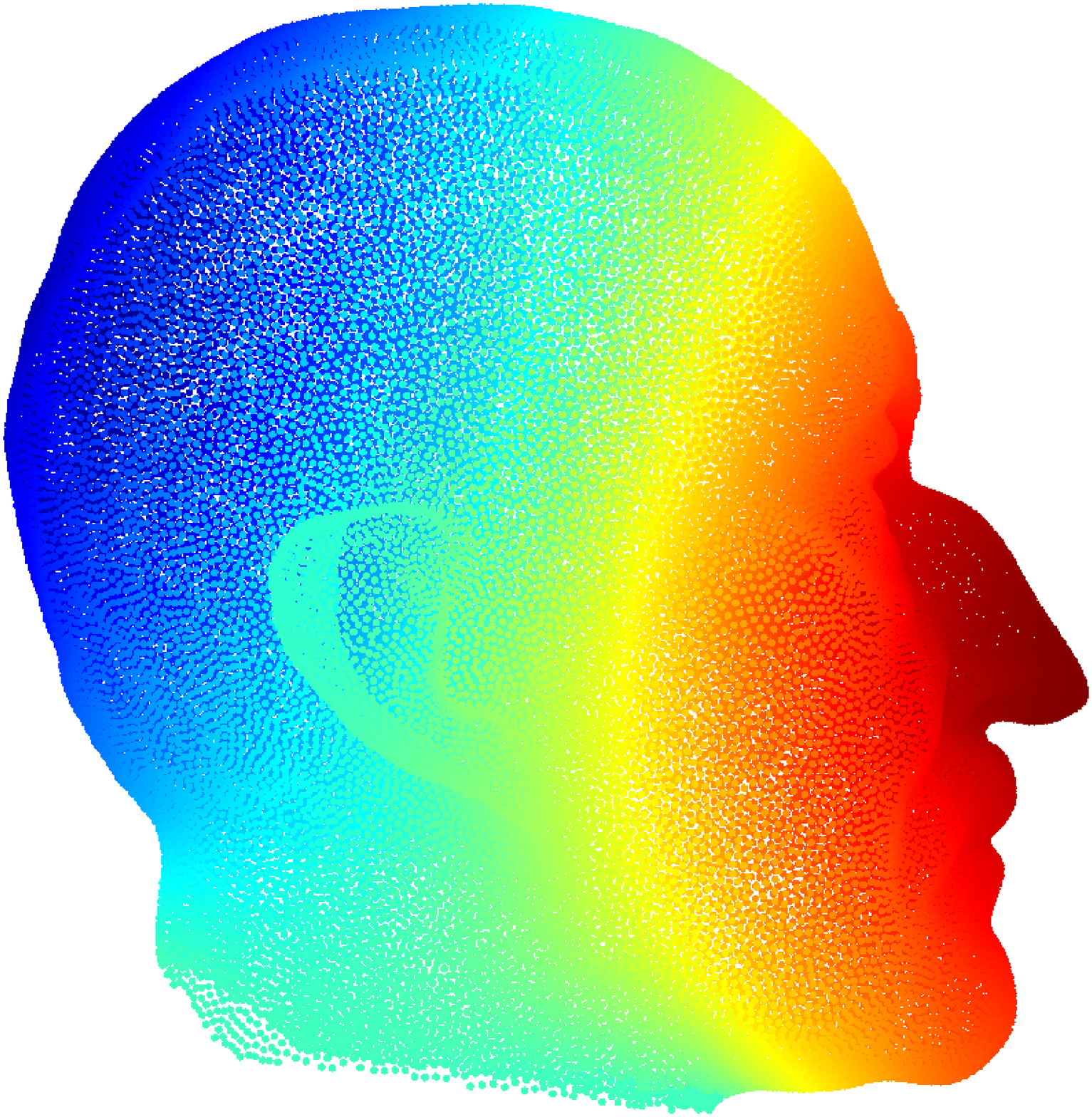} \\
\end{tabular}
\end{center}
\vspace{-4mm}
\caption{The first (upper row) and second (lower row) eigenfunctions with Dirichlet boundary.}
\label{fig:various_eigen}
\end{figure}

\section{Conclusion}
\label{sec:conclusion}

In this paper, we use the volume constraint \cite{Du-SIAM} in the Point Integral method to handle the Dirichlet boundary 
condition. And the convergence is proved both for Poisson equation and eigen problem of Laplace-Beltrami operator 
from point cloud. Our study shows that Point Integral method together with the volume constraint gives an efficient numerical 
approach to solve the Poisson equation with Dirichlet boundary on point cloud. In this paper, we focus on the Poisson equation. 
For other PDEs, we can also use the idea of volume constraint to enforce the Dirichlet boundary condition. The progress will be 
reported in our subsequent papers.

\vspace{0.2in}
\noindent
{\bf Acknowledgments.}
This research was partially supported by NSFC Grant (11201257 to Z.S., 11371220 to Z.S. and J.S. and 11271011 to J.S.), 
and  National Basic Research Program of China (973 Program 2012CB825500 to J.S.).

\appendix
\section{One basic estimates}

\begin{lemma}
  \label{lem:u-boundary}
Let $u(\bx)$ be the solution of \eqref{eq:dirichlet} and $f\in H^1(\M)$, then there is a generic constant $C>0$ and $T_0>0$ only depend on $\M$ and $\p\M$,
for any $t<T_0$, 
\begin{eqnarray}
  \int_{\V_t}|u(\by)|^2\mathd\by\le Ct^{3/2}\|f\|_{H^1(\M)}^2. \nonumber
\end{eqnarray}
\end{lemma}
\begin{proof} Both $\M$ and $\p \M$ are compact and $C^{\infty}$ smooth. Consequently, it is well known that both $\M$ and $\p \M$ have positive reaches, which means that 
there exists $T_0>0$ only depends on $\M$ and $\p\M$, if $t<T_0$,
$\V_t$ can be parametrized as $(\mathbf{z}(\by),\tau)\in \p\M\times [0,1]$, where $\by=\mathbf{z}(\by)+\tau(\mathbf{z}'(\by)-\mathbf{z}(\by))$
and $\left|\det\left(\frac{\mathd \by}{\mathd (\mathbf{z}(\by),\tau)}\right)\right|\le C\sqrt{t}$ and $C>0$ is a constant only depends on $\M$ and $\p\M$. 
Here $\mathbf{z}'(\by)$ is the intersection point between $\p\M'$ and the line determined by $\mathbf{z}(\by)$ and $\by$. The parametrization is illustrated in 
Fig.\ref{fig:boundary}.
  \begin{figure}
    \centering
    \includegraphics[width=0.6\textwidth]{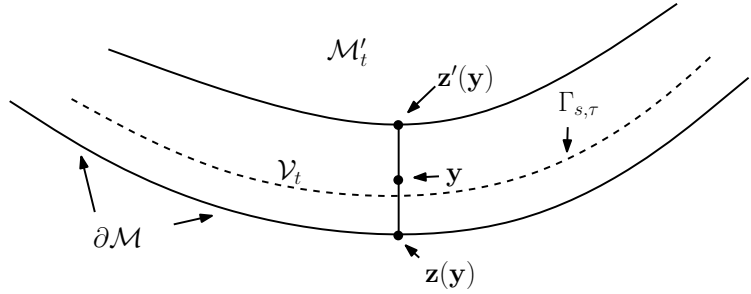}
    \caption{Parametrization of $\mathcal{V}_t$}
    \label{fig:boundary}
  \end{figure}

First, we have
  \begin{eqnarray}
&&  \int_{\V_t}|u(\by)|^2\mathd\by=\int_{\V_t}|u(\by)-u(\mathbf{z}(\by))|^2\mathd\by\nonumber\\
&=&\int_{\V_t}\left|\int_{0}^1\frac{\mathd }{\mathd s}u(\by+s(\mathbf{z}(\by)-\by))\mathd s\right|^2\mathd\by\nonumber\\
&=&\int_{\V_t}\left|\int_{0}^1(\mathbf{z}(\by)-\by)\cdot \nabla u(\by+s(\mathbf{z}(\by)-\by))\mathd s\right|^2\mathd\by\nonumber\\
&\le&Ct\int_{\V_t}\int_{0}^1\left| \nabla u(\by+s(\mathbf{z}(\by)-\by))\right|^2\mathd s\mathd\by\nonumber\\
&\le & Ct\sup_{0\le s\le 1}\int_{\V_t}\left| \nabla u(\by+s(\mathbf{z}(\by)-\by))\right|^2\mathd\by.\nonumber
\end{eqnarray}
Here, we use the fact that $\|\mathbf{z}(\by)-\by\|_2\le 2\sqrt{t}$ to get the second last inequality.

Then, the proof can be completed by following estimation.
\begin{eqnarray}
&&  \int_{\V_t}\left| \nabla u(\by+s(\mathbf{z}(\by)-\by))\right|^2\mathd\by\nonumber\\
&\le& C\sqrt{t}\int_0^1\int_{\p\M}\left| \nabla u(\mathbf{z}(\by)+(1-s)\tau(\mathbf{z}'(\by)-\mathbf{z}(\by)))\right|^2\mathd\mathbf{z}(\by)\mathd \tau\nonumber\\
&\le& C\sqrt{t}\sup_{0\le \tau\le 1}\int_{\p\M}\left| \nabla u(\mathbf{z}+(1-s)\tau(\mathbf{z}'-\mathbf{z}))\right|^2\mathd\mathbf{z}\nonumber\\
&\le& C\sqrt{t}\sup_{0\le \tau\le 1}\int_{\Gamma_{s,\tau}}\left| \nabla u(\tilde{\mathbf{z}})\right|^2\mathd\tilde{\mathbf{z}}\nonumber\\
&\le &C\sqrt{t} \|u\|_{H^2(\M)}^2\le C\sqrt{t} \|f\|_{H^1(\M)}^2,\nonumber
\end{eqnarray}
where $\Gamma_{s,\tau}$ is a $k-1$ dimensinal maniflod given by $\Gamma_{s,\tau}=\left\{\mathbf{z}+(1-s)\tau(\mathbf{z}'-\mathbf{z}):\mathbf{z}\in \p\M\right\}$.
We use the trace theorem to get the second last inequality and the last inequality is due to that $u$ is the solution of the Poisson equation \eqref{eq:dirichlet}.
\end{proof}

\input{discretization}

\bibliographystyle{abbrv}
\bibliography{poisson}


\end{document}

%% file: discretization.tex
\section{Proof of Theorem \ref{thm:dis_error}}

First, we need following important lemma which tells us that the discretized scheme is stable in $l^2$ sense.
\begin{lemma}
For any ${\bf u} = (u_1, \cdots, u_n)^t$ with $u_i=0,\;\bfp_i\in \V_t$, there exist constants $C>0,\, C_0>0$ independent on $t$ so that for sufficient small $t$ 
and $\frac{h}{\sqrt{t}}$ 
\begin{equation}
\frac{1}{t}\sum_{\bfp_i\in\M}\sum_{\bfp_j\in\M}R_t(\bfp_i,\bfp_j)(u_i-u_j)^2V_iV_j
\geq C(1-\frac{C_0h}{\sqrt{t}})\sum_{\bfp_j\in\M} u_i^2V_i.\nonumber
\end{equation}
\label{lem:elliptic_dis}
\end{lemma}

\begin{proof}

First, we introduce a smooth function $u$ that approximates $\bfu$ at the samples $P$. 
\begin{eqnarray}
  u(\bx)=\frac{1}{w_{t',h}(\bx)}\sum_{i=1}^n R_{t'}\left(\bx, \bfp_i\right)u_iV_i,\quad \bx\in \M,
\label{eqn:def_dis_u} 
\end{eqnarray}
where $w_{t',h}(\bx)=C_t\sum_{i=1}^nR\left(\frac{|\bx-\bfp_i|^2}{4t'}\right)V_i$ and $t'=t/18$.
Using the condition that $u_i=0,\; \bfp_i\in \V_t$ and $t'=t/18$, we know that
\begin{eqnarray}
  \label{eq:ub}
  u(\bx)=0,\quad \bx\in \V_{t'}.
\end{eqnarray}
Then using Lemma \ref{lem:elliptic_L_t}, we have
\begin{eqnarray}
  \int_{\M}|u(\bx)|^2\mathd \bx\le \frac{C}{t}\int_{\mathcal{M}}  \int_{\mathcal{M}}R_{t'}(\bx,\by) \left(u(\bx)-u(\by)\right)^2\mathd \mu_\bx \mathd\mu_\by.\nonumber
\end{eqnarray}
On the other hand
\begin{align}
\label{eqn:A1}
&\int_{\mathcal{M}}  \int_{\mathcal{M}}R_{t'}(\bx,\by) \left(u(\bx)-u(\by)\right)^2\mathd \mu_\bx \mathd\mu_\by\\
=& \int_{\mathcal{M}}  \int_{\mathcal{M}}R_{t'}(\bx,\by) \left(\frac{1}{w_{t',h}(\bx)}\sum_{i=1}^nR_{t'}(\bx,\bfp_i)u_iV_i
-\frac{1}{w_{t',h}(\by)}\sum_{j=1}^nR_{t'}(\bfp_j,\by)u_jV_j\right)^2\mathd \mu_\bx \mathd\mu_\by\nonumber\\
=& \int_{\mathcal{M}}  \int_{\mathcal{M}}R_{t'}(\bx,\by) \left(\frac{1}{w_{t',h}(\bx)w_{t',h}(\by)}\sum_{i,j=1}^nR_{t'}(\bx,\bfp_i)R_{t'}(\bfp_j,\by)V_i
V_j(u_i-u_j)\right)^2\mathd \mu_\bx \mathd\mu_\by\nonumber\\
\le & \int_{\mathcal{M}}  \int_{\mathcal{M}}R_{t'}(\bx,\by) \frac{1}{w_{t',h}(\bx)w_{t',h}(\by)}\sum_{i,j=1}^nR_{t'}(\bx,\bfp_i)R_{t'}(\bfp_j,\by)V_i
V_j(u_i-u_j)^2\mathd \mu_\bx \mathd\mu_\by\nonumber\\
=&  \sum_{i,j=1}^n\left(\int_{\mathcal{M}}  \int_{\mathcal{M}}\frac{1}{w_{t',h}(\bx)w_{t',h}(\by)}
R_{t'}(\bx,\bfp_i)R_{t'}(\bfp_j,\by)R_{t'}(\bx,\by)\mathd \mu_\bx \mathd\mu_\by\right)V_i
V_j(u_i-u_j)^2.\nonumber
\end{align}
Denote
$$A = \int_{\mathcal{M}}  \int_{\mathcal{M}}\frac{1}{w_{t',h}(\bx)w_{t',h}(\by)}
R_{t'}(\bx,\bfp_i)R_{t'}(\bfp_j,\by)R_{t'}(\bx,\by)\mathd \mu_\bx \mathd\mu_\by$$
and then notice only when $|\bfp_i-\bfp_j|^2\le 36t'$ is $A \neq 0$. For 
$|\bfp_i-\bfp_j|^2\le 36t'$, we have
\begin{align}
\label{eqn:A2}
A \le&\int_{\mathcal{M}}   \int_{\mathcal{M}}
R_{t'}(\bx,\bfp_i)R_{t'}(\bfp_j,\by)R_{t'}(\bx,\by)R\left(\frac{|\bfp_i-\bfp_j|^2}{72t'}\right)^{-1}R\left(\frac{|\bfp_i-\bfp_j|^2}{72t'}\right)\mathd \mu_\bx \mathd\mu_\by\\
\le&\frac{CC_t}{\delta_0}  \int_{\mathcal{M}}   \int_{\mathcal{M}}
R_{t'}(\bx,\bfp_i)R_{t'}(\bfp_j,\by)R\left(\frac{|\bfp_i-\bfp_j|^2}{72t'}\right)\mathd \mu_\bx \mathd\mu_\by\nonumber\\
\le&CC_t  \int_{\mathcal{M}}   \int_{\mathcal{M} }
R_{t'}(\bx,\bfp_i)R_{t'}(\bfp_j,\by)R\left(\frac{|\bfp_i-\bfp_j|^2}{72t'}\right)\mathd \mu_\bx \mathd\mu_\by\le CC_t R\left(\frac{|\bfp_i-\bfp_j|^2}{4t}\right).\nonumber
\end{align}

Combining Equation~\eqref{eqn:A1}, \eqref{eqn:A2} and Lemma~\ref{lem:elliptic_L_t}, we obtain
\begin{eqnarray}
\frac{C}{t}  \sum_{\bfp_i,\bfp_j\in \M}R_t\left(\bfp_i,\bfp_j\right)(u_i-u_j)^2V_iV_j\ge \int_{\mathcal{M}}  |u(\bx)|^2\mathd\mu_\bx
\label{eqn:u_2V_0}
\end{eqnarray}
 Denote
\begin{eqnarray}
B = \int_{\mathcal{M}}\frac{C_t}{w_{t',h}^2(\bx)}R\left(\frac{|\bx-\bfp_i|^2}{4t'}\right)
R\left(\frac{|\bx-\bfp_l|^2}{4t'}\right)\mathd\mu_\bx-\nonumber\\
\sum_{j=1}^n\frac{C_t}{w_{t',h}^2(\bfp_j)}R\left(\frac{|\bfp_j-\bfp_i|^2}{4t'}\right)R\left(\frac{|\bfp_j-\bfp_l|^2}{4t'}\right)V_j\nonumber
\end{eqnarray}
and then $|B|\le \frac{Ch}{t^{1/2}}$. At the same time, notice that only when $|\bfp_i-\bfp_l|^2 <16t'$ is $B\neq 0$. Thus we have 
\begin{eqnarray}
|B| \le \frac{1}{\delta_0} |A| R(\frac{|\bfp_i-\bfp_l|^2}{32t'}), \nonumber
\end{eqnarray}
and
\begin{eqnarray}
\label{eqn:u_2V_2}
&&\left|\int_{\mathcal{M}}u^2(\bx)\mathd\mu_\bx-\sum_{j=1}^nu^2(\bfp_j)V_j\right| \\
&\le&
  \sum_{i,l=1}^n|C_{t}u_iu_lV_iV_l| |A| \nonumber\\
&\le &\frac{Ch}{t^{1/2}} \sum_{i, l=1}^n\left|C_{t} R\left(\frac{|\bfp_i-\bfp_l|^2}{32t'}\right) u_iu_lV_iV_l \right|\nonumber\\
&\le &\frac{Ch}{t^{1/2}} \sum_{i, l=1}^n C_{t} R\left(\frac{|\bfp_i-\bfp_l|^2}{32t'}\right) u_i^2V_iV_l \le
\frac{Ch}{t^{1/2}} \sum_{i=1}^n u_i^2V_i.\nonumber
\end{eqnarray}
Now combining Equation~\eqref{eqn:u_2V_0} and \eqref{eqn:u_2V_2}, we have for small $t$ 
\begin{eqnarray}
  \sum_{i=1}^nu^2(\bfp_i)V_i &=& \int_{\mathcal{M}}  u^2(\bx)\mathd\mu_\bx+\frac{Ch}{t^{1/2}}\sum_{i=1}^n u_i^2V_i \nonumber \\
&\le & \frac{CC_t}{t}  \sum_{i,j=1}^nR\left(\frac{|\bfp_i-\bfp_j|^2}{4t}\right)(u_i-u_j)^2V_iV_j+ \frac{Ch}{t}\sum_{i=1}^n u_i^2V_i.\nonumber
\end{eqnarray}
Here we use the fact that for $t=18t'$ $$R\left(\frac{|\bfp_i-\bfp_j|^2}{4t'}\right) \leq \frac{1}{\delta_0} R\left(\frac{|\bfp_i-\bfp_j|^2}{4t}\right).$$


Let $\delta=\frac{w_{\min}}{2w_{\max}+w_{\min}}$. If $\sum_{i=1}^nu^2(\bfp_i)V_i\ge \delta^2 \sum_{i=1}^nu_i^2V_i$, then we have completed the proof. 
Otherwise, we have
\begin{eqnarray}
\sum_{i=1}^n( u_i-u(\bfp_i))^2V_i = \sum_{i=1}^n u_i^2V_i + \sum_{i=1}^n u(\bfp_i)^2V_i - 2 \sum_{i=1}^n u_iu(\bfp_i)V_i \ge (1- \delta)^2\sum_{i=1}^n u_i^2V_i. \nonumber
\end{eqnarray}
This enables us to prove the ellipticity of $\mathcal{L}$ in the case of $\sum_{i=1}^nu^2(\bfp_i)V_i < \delta^2 \sum_{i=1}^nu_i^2V_i$ as follows. 
\begin{eqnarray}
&&  C_t\sum_{i,j=1}^nR\left(\frac{|\bfp_i-\bfp_j|^2}{4t'}\right)(u_i-u_j)^2V_iV_j\nonumber\\
&=&   2C_t\sum_{i,j=1}^nR\left(\frac{|\bfp_i-\bfp_j|^2}{4t'}\right)u_i(u_i-u_j)V_iV_j\nonumber\\
&=& 2\sum_{i=1}^nu_i(u_i-u(\bfp_i))w_{t,h}(\bfp_i)V_i\nonumber\\
&=& 2\sum_{i=1}^n(u_i-u(\bfp_i))^2w_{t,h}(\bfp_i)V_i+2\sum_{i=1}^nu(\bfp_i)(u_i-u(\bfp_i))w_{t,h}(\bfp_i)V_i\nonumber\\
&\ge& 2\sum_{i=1}^n(u_i-u(\bfp_i))^2w_{t,h}(\bfp_i)V_i-2\left(\sum_{i=1}^nu^2(\bfp_i)w_{t,h}(\bfp_i)V_i\right)^{1/2}\left(\sum_{i=1}^n(u_i-u(\bfp_i))^2w_{t,h}(\bfp_i)V_i\right)^{1/2}\nonumber\\
&\ge& 2w_{\min}\sum_{i=1}^n(u_i-u(\bfp_i))^2V_i-2w_{\max}\left(\sum_{i=1}^nu^2(\bfp_i)V_i\right)^{1/2}\left(\sum_{i=1}^n(u_i-u(\bfp_i))^2V_i\right)^{1/2}\nonumber\\
&\ge& 2(w_{\min}(1-\delta)^2-w_{\max}\delta(1-\delta))\sum_{i=1}^nu_i^2V_i\ge w_{\min}(1-\delta)^2\sum_{i=1}^nu_i^2V_i.\nonumber
\end{eqnarray}

\end{proof}

One direct corollary of above lemma is the boundness of $\left(\sum_{\bfp_i\in \M}u_i^2V_i\right)^{1/2}$.
\begin{corollary}
Suppose $\mathbf{u}=(u_1, \cdots, u_n)^t$ with $u_i=0,\; \bfp_i\in \V_t$ solves the problem~\eqref{eq:dis} with $f\in C(\mathcal{M})$. 
Then there exists a constant $C>0$ such that 
\begin{eqnarray}
  \left(\sum_{\bfp_i\in \M}u_i^2V_i\right)^{1/2}\le C\|f\|_\infty,\nonumber
\end{eqnarray}
provided $t$ and $\frac{h}{\sqrt{t}}$ are small enough.
\label{thm:bound_solution_bfu}
\end{corollary}
\begin{proof}
From the elliptic property of $\mathcal{L}$, we have
\begin{eqnarray}
\sum_{\bfp_i\in \M} u_i^2 V_i &\leq&  \sum_{\bfp_i\in \M} \left( \sum_{\bfp_j \in \M} \rhkpipj f(\bfp_j)V_j \right) u_iV_i \nonumber \\
&\leq& \left(\sum_{\bfp_i\in \M}  u_i^2V_i\right)^{1/2} \left(\sum_{\bfp_i\in \M} 
\left( \|f\|_\infty \sum_{\bfp_j \in \M} \rhkpipj V_j \right)^2    V_i\right)^{1/2}\nonumber \\
&\leq& C\left(\sum_{\bfp_i\in \M}  u_i^2V_i\right)^{1/2} \|f\|_\infty.  \nonumber
\end{eqnarray}
This proves the lemma. 
\end{proof}

We can also get the bound of $u_{t,h}$ which is defined in \eqref{eq:interpolation} as following
\begin{eqnarray*}
  u_{t,h}(\bx)=\left\{\begin{array}{cc}\D
\frac{1}{w_{t,h}(\bx)}\left(\sum_{\bfp_j\in\M}R_t(\bx,\bfp_j)u_jV_j
 +t\sum_{\bfp_j\in\M}\bar{R}_t(\bx,\bfp_j)f(\bfp_j)V_j\right),&\bx\in \M'_t,\\
0,&\bx\in \V_t.
\end{array}\right.
\end{eqnarray*}

\begin{lemma}
  \label{lem:prior-u-dis}
Let $\bfu=[u_1,\cdots,u_n]^t$ be the solution of 
the problem~\eqref{eq:dis} with $f\in C(\M)$ and $u_{t,h}$ be associate smooth function defined in \eqref{eq:interpolation}.
 Then there exists $C>0$ such that 
\begin{eqnarray*}
  \left\|u_{t,h}(\bx)\right\|_{L^2(\M'_t)}&\le& C\|f\|_\infty,\\
\left\|\nabla u_{t,h}(\bx)\right\|_{L^2(\M'_t)}&\le& \frac{C}{\sqrt{t}}\|f\|_\infty
\end{eqnarray*}
\end{lemma}
\begin{proof}
  First, $\left\|u_{t,h}(\bx)\right\|_{L^2(\M'_t)}$ is bounded.
\begin{eqnarray}
&&\left\|u_{t,h}(\bx)\right\|_{L^2(\M'_t)}^2\nonumber\\
&=&\int_{\M'_t}\frac{1}{w_{t,h}^2(\bx)}\left(\sum_{\bfp_j\in P}R_t(\bx,\bfp_j)u_jV_j
-t\sum_{\bfp_j\in P}\bar{R}_t(\bx,\bfp_j)f_jV_j\right)^2\mathd\bx\nonumber\\
&\le & C\int_{\M'_t}\left(\sum_{\bfp_j\in P}R_t(\bx,\bfp_j)u_jV_j\right)^2\mathd\bx+
Ct^2\int_{\M'_t}\left(\sum_{\bfp_j\in P}\bar{R}_t(\bx,\bfp_j)f_jV_j\right)^2\mathd\bx\nonumber\\
&\le & C\int_{\M'_t}\left(\sum_{\bfp_j\in P}R_t(\bx,\bfp_j)V_j\right)\left(\sum_{\bfp_j\in P}R_t(\bx,\bfp_j)u_j^2V_j\right)\mathd\bx+
Ct^2\|f\|_\infty^2\int_{\M'_t}\left(\sum_{\bfp_j\in P}\bar{R}_t(\bx,\bfp_j)V_j\right)^2\mathd\bx\nonumber\\
&\le & C\sum_{\bfp_j\in P}u_j^2V_j\left(\int_{\M'_t}R_t(\bx,\bfp_j)\mathd\bx\right)+
Ct^2\|f\|_\infty^2\nonumber\\
&\le& C\sum_{\bfp_j\in P}u_j^2V_j+Ct^2\|f\|_\infty^2\le C\|f\|_\infty^2\nonumber
\end{eqnarray}
Using similar arguments, we can get the bound of $\left\|\nabla u_{t,h}(\bx)\right\|_{L^2(\M'_t)}$. From the definition of $u_{t,h}$,
we can see that all derivatives are applied on the kernel functions. The kernel functions are smooth functions, it gives one factor of $\frac{1}{\sqrt{t}}$
after derivative.
\end{proof}

Now, we are ready to prove Theorem \ref{thm:dis_error}.
\begin{proof} {\it of Theorem \ref{thm:dis_error}}

First, we split $L_{t} \left(u_{t,h} - u_{t}\right)$ to three terms, for any $\bx\in\M'_t$,
\begin{eqnarray}
\label{eq:error-dis-split}
&&\quad\quad L_{t} \left(u_{t,h} - u_{t}\right)\\
&=&  L_t(u_{t,h}) - L_{t, h}(u_{t, h})+L_{t, h}(u_{t, h})-L_tu_t\nonumber\\
&=&\left(L'_t(u_{t,h}) - L'_{t, h}(u_{t, h})\right)+\left(\frac{u_{t,h}(\bx)}{t}\left(\int_{\V_t} R_t(\bx,\by)\mathd\by-\sum_{\bfp_i\in \V_t}R_t(\bx,\bfp_i)V_i\right)\right)
\nonumber\\
&&+\left( \int_{\mathcal{M}}\bar{R}_t(\bx,\by)f(\by) - \sum_{\bfp_j\in \M}\bar{R}_t(\bx,\bfp_j)f(\bfp_j)V_j\right).\nonumber
\end{eqnarray}
To get the last equality, we use that $u_t$ and $u_{t,h}$ solve equation~\eqref{eq:integral} and equation~\eqref{eq:dis_interp} respectively.

The second and third terms are easy to bound. By using Lemma \ref{lem:prior-u-dis} and $(P,\mathbf{V})$ is $h$-integrable approximation of $\M$, we have
\begin{eqnarray}
\label{eq:dis-error-l2-0}
  \left\|\frac{u_{t,h}(\bx)}{t}\left(\int_{\V_t} R_t(\bx,\by)\mathd\by-\sum_{\bfp_i\in \V_t}R_t(\bx,\bfp_i)V_i\right)\right\|_{L^2(\M'_t)}
&\le& \frac{Ch}{t^{3/2}}\|f\|_\infty,\\
\label{eq:dis-error-dl2-0}
\quad\quad\quad\left\|\nabla \left(\frac{u_{t,h}(\bx)}{t}\left(\int_{\V_t} R_t(\bx,\by)\mathd\by-\sum_{\bfp_i\in \V_t}R_t(\bx,\bfp_i)V_i\right)\right)\right\|_{L^2(\M'_t)}
&\le& \frac{Ch}{t^2}\|f\|_\infty
\end{eqnarray}
and
\begin{eqnarray}
\label{eq:dis-error-l2-2}
  \left\|\int_{\mathcal{M}}\bar{R}_t(\bx,\by)f(\by) - \sum_{\bfp_j\in \M}\bar{R}_t(\bx,\bfp_j)f(\bfp_j)V_j\right\|_{L^2(\M'_t)}
&\le& \frac{Ch}{\sqrt{t}}\|f\|_{C^1(\M)},\\
\label{eq:dis-error-dl2-2}
\quad\quad\quad\left\|\nabla \left(\int_{\mathcal{M}}\bar{R}_t(\bx,\by)f(\by) - \sum_{\bfp_j\in \M}\bar{R}_t(\bx,\bfp_j)f(\bfp_j)V_j\right)\right\|_{L^2(\M'_t)}
&\le& \frac{Ch}{t}\|f\|_{C^1(\M)}
\end{eqnarray}
The first term of \eqref{eq:error-dis-split} is much more complicated to bound. We split it further to two terms. Denote
\begin{eqnarray}
  a_{t,h}(\bx)&=&\frac{1}{w_{t,h}(\bx)}\sum_{\bfp_j\in \M'_t}R_t(\bx,\bfp_j)u_jV_j,\\
c_{t,h}(\bx)&=&\frac{t}{w_{t,h}(\bx)}\sum_{\bfp_j\in \M'_t}\bar{R}_t(\bx,\bfp_j)f(\bfp_j)V_j, 
\end{eqnarray}
and then $u_{t,h}(\bx) = a_{t, h}(\bx) +c_{t, h}(\bx),\; \bx\in \M'_t$. 

First we upper bound $\|L'_t(u_{t,h}) - L'_{t, h}(u_{t, h})\|_{L^2(\M'_t)}$.
For $c_{t, h}$, we have 
\begin{eqnarray}
&&\left|\left(L'_tc_{t, h} - L'_{t, h}c_{t, h}\right)(\bx)\right|\nonumber \\
&=& \frac{1}{t} \left|\int_{\mathcal{M}'}R_t(\bx,\by)(c_{t,h}(\bx)-c_{t,h}(\by))  \mathd \mu_\by-\sum_{\bfp_j\in \M'_t}R_t(\bx,\bfp_j)(c_{t,h}(\bx)-c_{t,h}(\bfp_j))V_j\right|\nonumber\\
&\le &\frac{1}{t} \left|c_{t,h}(\bx)\right|\left|\int_{\mathcal{M}'}R_t(\bx,\by) \mathd \mu_\by-\sum_{\bfp_j\in \M'_t}R_t(\bx,\bfp_j)V_j\right|\nonumber\\
&&+ \frac{1}{t}\left|\int_{\mathcal{M}'}R_t(\bx,\by)c_{t,h}(\by)  \mathd \mu_\by-\sum_{\bfp_j\in \M'_t}R_t(\bx,\bfp_j)c_{t,h}(\bfp_j)V_j\right|\nonumber\\
&\le &\frac{Ch}{t^{3/2}}\left|c_{t,h}(\bx)\right|+ \frac{Ch}{t^{3/2}} \|c_{t,h}\|_{C^1(\M'_t)}\nonumber\\
&\le &\frac{Ch}{t^{3/2}}t\|f\|_\infty+ \frac{Ch}{t^{3/2}}(t\|f\|_\infty + t^{1/2}\|f\|_{\infty})\le \frac{Ch}{t}\|f\|_{\infty}.\nonumber
\end{eqnarray}
For $a_{t, h}$, we have 
\begin{eqnarray}
\label{eqn:a_1}
&&  \int_{\mathcal{M}'}\left(a_{t,h}(\bx)\right)^2\left|\int_{\mathcal{M}'}R_t(\bx,\by) \mathd \mu_\by-\sum_{\bfp_j\in \M'_t}R_t(\bx,\bfp_j)V_j\right|^2\mathd\mu_\bx\\
&\le & \frac{Ch^2}{t}\int_{\mathcal{M}'}\left(a_{t,h}(\bx)\right)^2\mathd\mu_\bx \nonumber\\
&\le & \frac{Ch^2}{t} \int_{\mathcal{M}'} \left( \frac{1}{w_{t, h}(\bx)} \sum_{\bfp_j\in \M'_t}R_t(\bx,\bfp_j)u_jV_j \right)^2 \mathd\mu_\bx \nonumber \\
&\le & \frac{Ch^2}{t} \int_{\mathcal{M}'} \left( \sum_{\bfp_j\in P}R_t(\bx,\bfp_j)u_j^2V_j  \right) \left( \sum_{\bfp_j\in \M'_t}R_t(\bx,\bfp_j)V_j  \right) \mathd\mu_\bx \nonumber \\
&\le & \frac{Ch^2}{t} \left( \sum_{\bfp_j\in \M'_t}u_j^2V_j \int_{\mathcal{M}'}R_t(\bx,\bfp_j)  \mathd\mu_\bx    \right)  \le  \frac{Ch^2}{t}\sum_{\bfp_j\in \M'_t}u_j^2V_j.\nonumber
\end{eqnarray}
Let 
\begin{eqnarray}
A &=&  C_t\int_{\mathcal{M}'}\frac{1}{w_{t, h}(\by)}R\left(\frac{|\bx-\by|^2}{4t}\right)R\left(\frac{|\bfp_i-\by|^2}{4t}\right) \mathd \mu_\by\nonumber\\
 &-&C_t\sum_{\bfp_j\in \M'_t} \frac{1}{w_{t, h}(\bfp_j)}R\left(\frac{|\bx-\bfp_j|^2}{4t}\right)R\left(\frac{|\bfp_i-\bfp_j|^2}{4t}\right)V_j. \nonumber
\end{eqnarray}
We have $|A|<\frac{Ch}{t^{1/2}}$ for some constant $C$ independent of $t$. In addition, notice that
only when $|\bx-\bfp_i|^2\leq 16t $ is $A\neq 0$, which implies 
\begin{eqnarray}
|A| \leq \frac{1}{\delta_0}|A|R\left(\frac{|\bx-\bfp_i|^2}{32t}\right). \nonumber
\end{eqnarray}
Then we have
\begin{align}
\label{eqn:a_2}
&\int_{\mathcal{M}'}\left|\int_{\mathcal{M}'}R_t(\bx,\by)a_{t,h}(\by)  \mathd \mu_\by-\sum_{\bfp_j\in \M'_t}R_t(\bx,\bfp_j)a_{t,h}(\bfp_j)V_j\right|^2\mathd\mu_\bx\\
=& \int_{\mathcal{M}'}\left(\sum_{\bfp_i\in \M'_t}C_tu_iV_i A \right)^2\mathd\mu_\bx\nonumber\\
\le &\frac{Ch^2}{t} \int_{\mathcal{M}'}\left(\sum_{\bfp_i\in \M'_t} C_t|u_i|V_i R\left(\frac{|\bx-\bfp_i|^2}{32t}\right)  \right)^2 \mathd\mu_\bx \nonumber\\
\le &\frac{Ch^2}{t} \int_{\mathcal{M}'} \left(\sum_{\bfp_i\in \M'_t} C_t R\left(\frac{|\bx-\bfp_i|^2}{32t}\right) u^2_iV_i\right)  
\left(\sum_{\bfp_i\in \M'_t} C_t R\left(\frac{|\bx-\bfp_i|^2}{32t} \right)V_i  \right) \mathd\mu_\bx \nonumber\\
\le &\frac{Ch^2}{t} \sum_{\bfp_i\in \M'_t} u^2_iV_i\left(\int_{\mathcal{M}'} C_t R\left(\frac{|\bx-\bfp_i|^2}{32t}\right)  \mathd\mu_\bx \right)  \le
 \frac{Ch^2}{t} \left(\sum_{\bfp_i\in \M'_t}u_i^2V_i\right). \nonumber
\end{align} 
Combining Equation~\eqref{eqn:a_1},~\eqref{eqn:a_2} and Theorem~\ref{thm:bound_solution_bfu}, 
\begin{eqnarray}
&&\|L'_ta_{t, h} - L'_{t, h}a_{t, h}\|_{L^2(\M)} \nonumber \\
&=&\left(\int_{\M'_t} \left|\left(L'_t(a_{t, h}) - L'_{t, h}(a_{t, h})\right)(\bx)\right|^2 \mathd\mu_\bx\right)^{1/2} \nonumber\\
&\le & \frac{Ch}{t^{3/2}} \left(\sum_{\bfp_i\in \M'_t}u_i^2V_i\right)^{1/2} \leq  \frac{Ch}{t^{3/2}}\|f\|_\infty.\nonumber
\end{eqnarray}
Assembling the parts together, we have the following upper bound.
\begin{eqnarray}
\label{eq:dis-error-l2-1}
&&\|L'_tu_{t, h} - L'_{t, h}u_{t, h}\|_{L^2(\M'_t)}\\
&\le& \|L'_ta_{t, h} - L'_{t, h}a_{t, h}\|_{L^2(\M'_t)}  + \|L'_tc_{t, h} - L'_{t, h}c_{t, h}\|_{L^2(\M'_t)} \nonumber \\
&\le& \frac{Ch}{t^{3/2}}\|f\|_\infty + \frac{Ch}{t}\|f\|_{\infty} \le \frac{Ch}{t^{3/2}} \|f\|_\infty.\nonumber 
\end{eqnarray}

The complete $L^2$ estimate follows from Equation~\eqref{eq:dis-error-l2-0},~\eqref{eq:dis-error-l2-2} and~\eqref{eq:dis-error-l2-1}. 

Next, we turn to upper bound $\|\nabla (L'_tu_{t}-L'_{t,h}u_{t, h})\|_{L^2(\M'_t)}$. 
Consider $\|\nabla (L'_ta_{t, h} - L'_{t, h}a_{t, h})\|_{L_2(\M'_t)}$, it can be splited into the summation of three terms. Next, we estimate these three terms separately. 
The first term is
\begin{eqnarray}
\label{eqn:da_1}
&&  \int_{\mathcal{M}'_t}\left|\nabla a_{t,h}(\bx)\right|^2\left|\int_{\mathcal{M}'_t}R_t(\bx,\by) \mathd \mu_\by-\sum_{\bfp_j\in \M'_t}R_t(\bx,\bfp_j)V_j\right|^2\mathd\mu_\bx\\
&\le& \frac{Ch^2}{t}\int_{\mathcal{M}'_t}\left|\nabla a_{t,h}(\bx)\right|^2\mathd\mu_\bx \nonumber\\
&\le & \frac{Ch^2}{t} \left(\int_{\mathcal{M}'_t} \left| \frac{1}{w_{t, h}(\bx)} \sum_{\bfp_j\in \M'_t}\nabla R_t(\bx,\bfp_j)u_jV_j \right|^2 \mathd\mu_\bx\right.\nonumber\\
&&\left.+\int_{\mathcal{M}'_t} 
\left| \frac{\nabla w_{t,h}(\bx)}{w^2_{t, h}(\bx)} \sum_{\bfp_j\in \M'_t} R_t(\bx,\bfp_j)u_jV_j \right|^2 \mathd\mu_\bx\right)\nonumber\\
&\le & \frac{Ch^2}{t^2} \int_{\mathcal{M}'_t} \left| \sum_{\bfp_j\in \M'_t} R_{2t}(\bx,\bfp_j)u_jV_j \right|^2 \mathd\mu_\bx\nonumber\\
&\le & \frac{Ch^2}{t^2} \int_{\mathcal{M}'_t} \left( \sum_{\bfp_j\in \M'_t}R_{2t}(\bx,\bfp_j)u_j^2V_j  \right) 
\left( \sum_{\bfp_j\in \M'_t}R_{2t}(\bx,\bfp_j)V_j  \right) \mathd\mu_\bx \nonumber \\
&\le & \frac{Ch^2}{t^2} \left( \sum_{\bfp_j\in \M'_t}u_j^2V_j \int_{\mathcal{M}'_t}R_{2t}(\bx,\bfp_j)  \mathd\mu_\bx    \right)  
\le  \frac{Ch^2}{t^2}\sum_{\bfp_j\in \M'_t}u_j^2V_j.
\nonumber
\end{eqnarray}
where $R_{2t}(\bx,\bfp_j)=C_tR\left(\frac{|\bx-\bfp_j|^2}{8t}\right)$. Here we use the assumption that $R(s)>\delta_0$ for all 
$0\le s\le 1/2$.

The second term is
\begin{eqnarray}
\label{eqn:da_2}
&&  \int_{\mathcal{M}'_t}\left| a_{t,h}(\bx)\right|^2\left|\int_{\mathcal{M}'_t}\nabla R_t(\bx,\by) \mathd \mu_\by-\sum_{\bfp_j\in \M'_t}\nabla R_t(\bx,\bfp_j)V_j\right|^2\mathd\mu_\bx\\
&\le & \frac{Ch^2}{t^2}\int_{\mathcal{M}'_t}\left| a_{t,h}(\bx)\right|^2\mathd\mu_\bx \le  \frac{Ch^2}{t^2}\sum_{\bfp_j\in \M'_t}u_j^2V_j.\nonumber
\end{eqnarray}
Let 
\begin{eqnarray}
B &=&  C_t\int_{\mathcal{M}'_t}\frac{1}{w_{t, h}(\by)}\nabla R\left(\frac{|\bx-\by|^2}{4t}\right)R\left(\frac{|\bfp_i-\by|^2}{4t}\right) \mathd \mu_\by\nonumber\\
 &-&C_t\sum_{\bfp_j\in \M'_t} \frac{1}{w_{t, h}(\bfp_j)}\nabla R\left(\frac{|\bx-\bfp_j|^2}{4t}\right)R\left(\frac{|\bfp_i-\bfp_j|^2}{4t}\right)V_j. \nonumber
\end{eqnarray}
We have $|B|<\frac{Ch}{t^{1/2}}$ for some constant $C$ independent of $t$. In addition, notice that
only when $|\bx-\bx_i|^2\leq 16t $ is $B\neq 0$, which implies 
\begin{eqnarray}
|B| \leq \frac{1}{\delta_0}|B|R\left(\frac{|\bx-\bfp_i|^2}{32t}\right). \nonumber
\end{eqnarray}
Then we have the upper bound of the third term
\begin{eqnarray}
\label{eqn:da_3}
&&\int_{\mathcal{M}'_t}\left|\int_{\mathcal{M}'_t}\nabla R_t(\bx,\by)a_{t,h}(\by)  \mathd \mu_\by-\sum_{\bfp_j\in \M'_t}
\nabla R_t(\bx,\bfp_j)a_{t,h}(\bfp_j)V_j\right|^2\mathd\mu_\bx\\
&=& \int_{\mathcal{M}'_t}\left(\sum_{\bfp_i\in \M'_t}C_tu_iV_i B \right)^2\mathd\mu_\bx\nonumber\\
&\le &\frac{Ch^2}{t^2} \int_{\mathcal{M}'_t}\left(\sum_{\bfp_i\in \M'_t} C_t|u_i|V_i R\left(\frac{|\bx-\bfp_i|^2}{32t}\right)  \right)^2 \mathd\mu_\bx \nonumber\\
&\le & \frac{Ch^2}{t^2} \left(\sum_{\bfp_i\in \M'_t}u_i^2V_i\right). 
\nonumber
\end{eqnarray} 
Combining Equation~\eqref{eqn:da_1}, ~\eqref{eqn:da_2} and ~\eqref{eqn:da_3}, we have
\begin{eqnarray}
&&\|\nabla (L'_ta_{t, h} - L'_{t, h}a_{t, h})\|_{L^2(\M'_t)} \nonumber \\
&=&\left(\int_{\M'_t} \left|\left(L_t(a_{t, h}) - L_{t, h}(a_{t, h})\right)(\bx)\right|^2 \mathd\mu_\bx\right)^{1/2} \nonumber\\
&\le & \frac{Ch}{t^{2}}\left(\sum_{\bfp_i\in \M'_t}u_i^2V_i\right)^{1/2}\le \frac{Ch}{t^{2}}\|f\|_{\infty}
\nonumber
\end{eqnarray}
Using a similar argument, we obtain
\begin{eqnarray}
\|\nabla (L'_tc_{t, h} - L'_{t, h}c_{t, h})\|_{L^2(\M'_t)} &\le & \frac{Ch}{t^{3/2}}\|f\|_\infty,\nonumber
\end{eqnarray}
and thus
\begin{eqnarray}
\|\nabla (L'_tu_{t, h} - L'_{t, h}u_{t, h})\|_{L^2(\M'_t)} \le \frac{Ch}{t^{2}}\|f\|_{\infty}.
\label{eq:dis-error-dl2-1}
\end{eqnarray}
At last, we complete the proof using \eqref{eq:dis-error-dl2-0}, \eqref{eq:dis-error-dl2-2} and \eqref{eq:dis-error-dl2-1}
\end{proof}
